\documentclass[final,leqno,onefignum,onetabnum]{siamltex1213}

\usepackage{amssymb}
\usepackage{amsmath}
\usepackage{subfig}

\newcommand{\norm}[1]{\lVert#1\rVert}
\newcommand{\abs}[1]{\lvert#1\rvert} 
\newcommand{\inner}[1]{\langle#1\rangle} 

\newcommand{\mta}{{\mathcal{M}_a}}
\newcommand{\redel}{\mathop{\textup{Re}}}
\newcommand{\imdel}{\mathop{\textup{Im}}}

\newcommand{\mspan}{\mathop{\textup{span}}}
\newcommand{\dist}{\mathop{\textup{dist}}}
\newcommand{\essinf}{\mathop{\textup{ess\,inf}}}
\newcommand{\ident}{\mathop{\textup{Id}}}
\newcommand{\nd}{\mathcal{R}}

\newcommand{\rum}[1]{\mathbb{#1}}

\newcommand{\Sref}[1]{Section~\ref{#1}}

\newcommand{\sang}{\zeta}

\newtheorem{remark}{Remark}

\title{Distinguishability revisited: depth dependent bounds on reconstruction quality in electrical impedance tomography\thanks{This research is supported by Advanced Grant 291405 HD-Tomo from the European Research Council.}}

\author{Henrik Garde\footnotemark[2] \footnotemark[3] \and Kim Knudsen\footnotemark[2]}

\begin{document}
\maketitle

\renewcommand{\thefootnote}{\fnsymbol{footnote}}
\footnotetext[2]{Department of Applied Mathematics and Computer Science, Technical University of Denmark, 2800 Kgs. Lyngby, Denmark.}
\footnotetext[3]{Department of Mathematical Sciences, Aalborg University, 9100 Aalborg, Denmark.}
\renewcommand{\thefootnote}{\arabic{footnote}}

\slugger{siap}{xxxx}{xx}{x}{x--x}

\begin{abstract}
The reconstruction problem in electrical impedance tomography is highly ill-posed, and it is often observed numerically that reconstructions have poor resolution far away from the measurement boundary but better resolution near the measurement boundary. The observation can be quantified by the concept of distinguishability of inclusions. This paper provides mathematically rigorous results supporting the intuition. Indeed, for a model problem lower and upper bounds on the distinguishability of an inclusion are derived in terms of the boundary data. These bounds depend explicitly on the distance of the inclusion to the boundary, i.e.\ the depth of the inclusion. The results are obtained for disk inclusions in a homogeneous background in the unit disk. The theoretical bounds are verified numerically using a novel, exact characterization of the forward map as a tridiagonal matrix. 
\end{abstract}



\pagestyle{myheadings}
\thispagestyle{plain}
\markboth{H. GARDE AND K. KNUDSEN}{BOUNDS ON RECONSTRUCTION QUALITY IN EIT}

\section{Introduction}

The goal of electrical impedance tomography (EIT) is to reconstruct
the internal electrical conductivity of an object. This is done from
voltage and current boundary measurements through electrodes on the object's
surface. Applications of EIT include, among others, monitoring patient lung function, geophysics, and industrial tomography for instance for non-destructive imaging of cracks in concrete \cite{Holder2005,Abubakar2009,York2001,Cheney1999,Hanke2003,Uhlmann2009,Karhunen2010,Karhunen2010a}. For a given EIT device with fixed precision and measurements corrupted by noise it is of course important to have a basic understanding of the quality and reliability of reconstructed conductivities. There seems to be a well-established intuition that details further away from the measurement boundary are more difficult to reconstruct reliably than details closer to the boundary, i.e.\ the resolution in reconstructions is depth dependent.

The inverse problem in EIT is highly ill-posed, and under reasonable assumptions it is possible to obtain conditional log-type stability estimates \cite{Alessandrini1988,Mandache2001}. It is worth noting that these estimates are uniform throughout the domain and therefore do not take into account the distance to the boundary. In spite of the global estimates, reconstruction algorithms often produce good results close to the boundary (e.g.\ \cite{Garde_2015,Knudsen_2015,sparse3d,winkler2014}). For a specific example see Figure \ref{fig:fig1} where an inclusion is more accurately reconstructed close to the measurement boundary. No theoretical results seem to address this depth dependence in general; for the linearized problem, however, a few results exist \cite{Nagayasu_2009,Ammari_2013}. The main results presented in this paper will for the first time provide theoretical evidence for the non-linear problem.

\begin{figure}[htb]
	\centering
	\includegraphics[width = .45\textwidth]{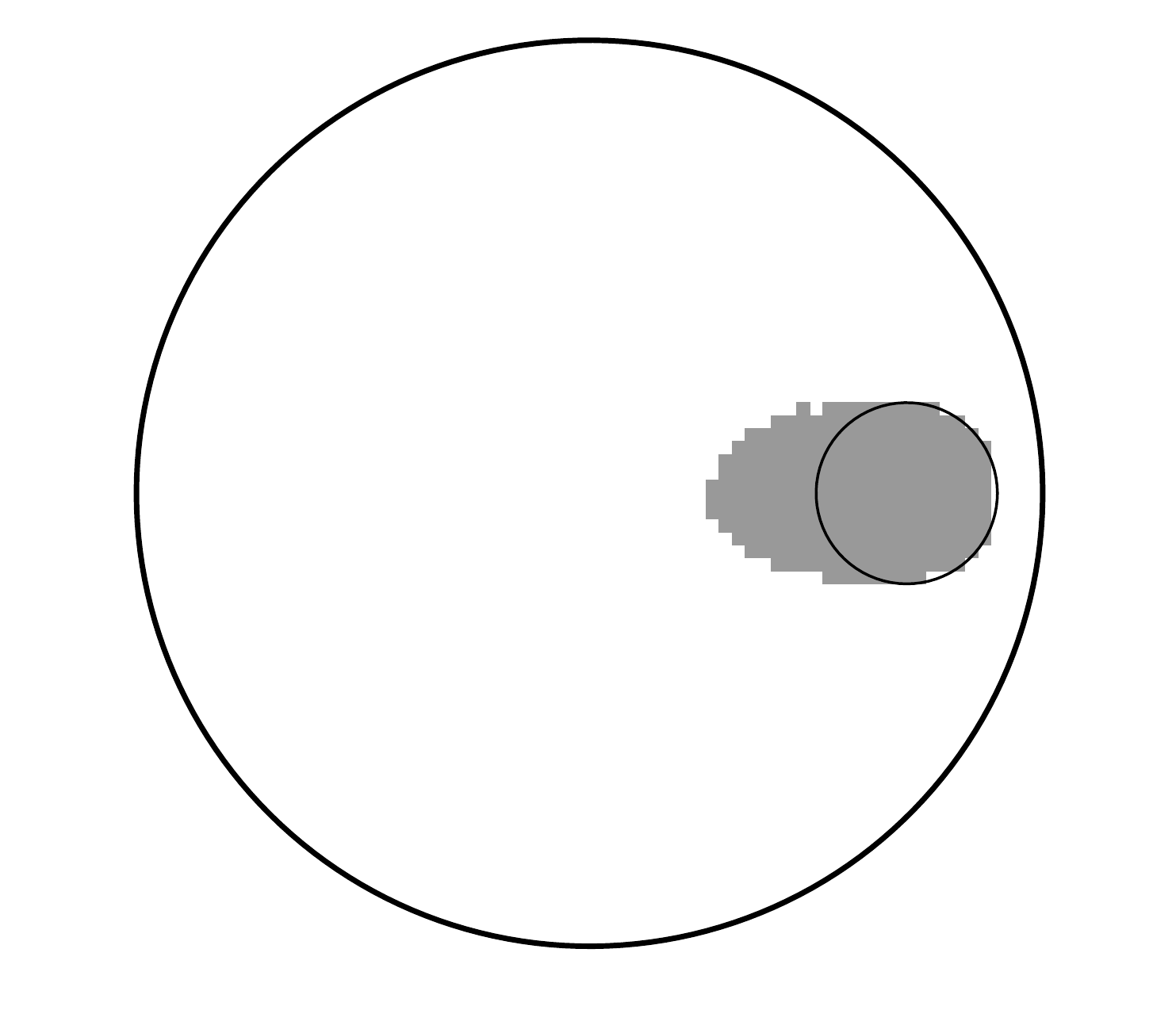}
	\caption{Reconstruction in the unit disk of a ball inclusion $B$ (black outline) with center $(0.7,0)$ and radius $0.2$, by use of the monotonicity method (cf.\ \cite{Harrach13,GardeStaboulis_2016}). The simulated data is based on the conductivity $\gamma = 1 + 4\chi_{B}$ and 32 trigonometric current patterns. Noise is added corresponding to a 0.5\% noise level.} \label{fig:fig1}
\end{figure}

Given the real-valued conductivity
\begin{equation*}
	\gamma\in L^\infty_+(\Omega) \equiv \{w\in L^\infty(\Omega) : \essinf w > 0\}
\end{equation*}
the forward problem of EIT is governed by the \emph{conductivity equation}%
\begin{equation}
	\nabla \cdot(\gamma\nabla u) = 0, \text{ in } \Omega, \label{eq:condeq}
\end{equation} 
where $u$ models the interior electric potential and $\Omega\subset \rum{R}^d$ is a bounded Lipschitz domain for $d \geq 2$ modelling the physical object. Depending on the choice of boundary conditions various models for EIT arise. The simplest model is Calder\'on's original formulation of the \emph{continuum model} \cite{Calder'on1980} that given a boundary potential $f\in H^{1/2}(\partial\Omega)$ makes use of the Dirichlet boundary condition
\begin{equation*}
	u|_{\partial\Omega} = f \text{ on } \partial\Omega,
\end{equation*}
where $u|_{\partial\Omega}$ denotes the trace of $u$ to the boundary $\partial \Omega.$ Standard elliptic theory for the continuum model gives a unique solution $u\in H^1(\Omega)$, and the resulting  boundary current flux is then given by $\nu\cdot\gamma\nabla u|_{\partial\Omega}$ with $\nu$ denoting the outward unit normal to $\partial\Omega.$ All possible boundary measurements are encoded in the Dirichlet-to-Neumann (DN) map defined by 
\begin{align*}
  \Lambda(\gamma)\colon H^{1/2}(\partial\Omega)&\to H^{-1/2}(\partial\Omega),\\
u|_{\partial\Omega} &\mapsto \nu\cdot\gamma\nabla u,
\end{align*}
and the inverse problem of EIT is thus to reconstruct $\gamma$ given $\Lambda(\gamma).$  Uniqueness for the inverse problem with the continuum model is a well-studied topic \cite{Sylvester1987, Nachman1996, Haberman_2013, CaroRogers2016}; we focus on 2D where there is uniqueness for general conductivities in $L^\infty_+(\Omega)$ if the domain is simply connected \cite{Astala2006a}.

In this paper we consider the domain $\Omega$ to be the unit disk
$\rum{D}\equiv \{x\in\rum{R}^2 : \abs{x}< 1 \}$ with conductivities
defined by a uniform background with one circular inclusion. This is
certainly a simplification in comparison to real measurement
scenarios, but the ideal model allows an explicit understanding of the
depth dependence that may shed light upon more complex situations. Let $A>-1$ and let $\chi_{B_{C,R}}$ be the characteristic function on
the open ball $B_{C,R}\subset \rum{D}$ with centre $C$ and radius
$R,$ and define the model conductivity $\gamma = 1+A\chi_{B_{C,R}}.$ Suppose we have a DN map contaminated by noise, i.e.\ $\Lambda^\delta \equiv \Lambda(1+A\chi_{B_{C,R}}) + E^\delta$ with a noise level $\norm{E^\delta}_{\mathcal{L}(L^2(\partial\rum{D}))} = \delta$. To ensure that $\Lambda^\delta$ contains information about the inclusion 
we need 
\begin{equation}
	\norm{\Lambda(1+A\chi_{B_{C,R}})-\Lambda(1)}_{\mathcal{L}(L^2(\partial\rum{D}))} \label{eq:distin}
\end{equation}
to be larger than $\delta,$ and hence we call \eqref{eq:distin} the \emph{distinguishability} of the inclusion $B_{C,R}$ with contrast $A$ to the background. $\mathcal{L}(L^2(\partial\rum{D}))$ in \eqref{eq:distin} denotes the space of bounded linear operators from $L^2(\partial\rum{D})$ to itself. The difference operator $\Lambda(1+A\chi_{B_{C,R}})-\Lambda(1)$ is compact and self-adjoint in $L^2(\partial\rum{D})$ (cf.\ Lemma \ref{lemma:DNL2bnd}), so the norm in \eqref{eq:distin} equals the largest magnitude eigenvalue of $\Lambda({1+A\chi_{B_{C,R}}})-\Lambda(1)$. 

In \cite{Isaacson1986,Cheney1992} the norm
\begin{equation}
  \label{eq:distin2}
	\norm{\nd(1 + A\chi_{B_{0,r}})-\nd(1))}_{\mathcal{L}(L_\diamond^2(\partial\rum{D}))}  
\end{equation}
was used to define distinguishability. Here $\mathcal{R}(\gamma)$ denotes the Neumann-to-Dirichlet (ND) map (the inverse of $\Lambda(\gamma)$) and $B_{0,r}$ a concentric ball with radius $r$. The characterization of \eqref{eq:distin2} is straightforward, as the eigenvalues of the operator $\nd(1 + A\chi_{B_{0,r}})-\nd(1)$ can be found explicitly by separation of variables. 

In contrast to \cite{Isaacson1986,Cheney1992} we use non-concentric balls $B_{C,R}.$ As a consequence we do not get a full characterization of \eqref{eq:distin} but rather explicit lower and upper bounds (Theorem~\ref{dnbounds}), which depend on the distance of $B_{C,R}$ to the boundary, i.e.\ the depth of the inclusion. The bounds show that the distinguishability is decreasing with the depth of the inclusion, and that the distinguishability can be arbitrarily high when the inclusion is sufficiently close to the boundary. Furthermore, the depth dependence can be formulated for inclusions of fixed size but varying distance to the boundary (cf.\ Corollary \ref{coro:fixedsize}).

The spectrum of $\Lambda(1+A\chi_{B_{C,R}})-\Lambda(1)$ does in general not have a known explicit characterization, but in case of a non-concentric inclusion it can be related to the known spectrum of a concentric inclusion  by the use of M\"obius transformations. These transformations belong to a class of harmonic morphisms that is used widely in EIT for instance in reconstruction \cite{Hanke_2008,Ikehata2000a,Kress_2011,Kress_2012,Akduman_2002,Saka_2004}, and recently for generating spatially varying meshes trying to accommodate for the depth dependence in numerical reconstruction when using electrode models \cite{winkler2014}. 

Before describing the general structure of the paper, we give a few comments on the simplifications used to obtain the distinguishability bounds, and the possible application of the bounds to real measurement scenarios. The unit disk domain is a natural choice of domain both in terms of depth dependence, as it is rotationally symmetric, but also in terms of the Riemann mapping theorem (e.g.\ \cite{Stein2003}) which states that simply connected domains in $\rum{C}$ are conformally equivalent to the unit disk. If we consider an open set $\mathcal{D}$ as the inclusion, we may pick open balls $B_1$ and $B_2$ such that $B_1\subseteq \mathcal{D}\subseteq B_2$. The distinguishability of $\mathcal{D}$ can then be related to the presented results in this paper using the monotonicity relations outlined in Appendix \ref{sec:appA}. For practical measurements there are also other forward models for EIT that can reduce modelling errors, such as the \emph{complete electrode model} (CEM) \cite{Somersalo1992}. However, in \cite{Hyvonen09,GardeStaboulis_2016} it was proved that the difference in the forward map of CEM and the continuum model, as well as their Fr\'echet derivatives, depends linearly on a parameter that characterizes how densely the electrodes cover the boundary. It is therefore expected that, for sufficiently many equidistant boundary electrodes, any depth dependent properties of the continuum model will also be observed for the CEM.

In the rest of the paper $(x_1,x_2)\in\rum{R}^2$ will be identified with $x_1+ix_2\in\rum{C}$. Furthermore, $\norm{\cdot}$ will denote the $L^2(\partial\rum{D})$-norm and $\inner{\cdot,\cdot}$ the corresponding inner product.

The paper is organised as follows: in Section~\ref{sec:mobdn} we introduce M\"obius transformations in the unit disk, and the DN map for non-concentric inclusions is given in terms of these transformations. The distinguishability bounds are derived in Theorem~\ref{dnbounds} in Section~\ref{sec:dnbnds}. Section~\ref{sec:bnds} gives an exact tridiagonal matrix representation of the non-concentric DN maps to accurately and efficiently validate the bounds numerically and demonstrate their tightness. Finally, we conclude in Section~\ref{sec:conc}.

In Appendix \ref{sec:appB} results regarding bounds on distinguishability and exact matrix characterization for the Neumann-to-Dirichlet (ND) map are given. While the actual bounds for the ND map are fundamentally different from the DN counterparts, they are placed in the appendix due to the nature of the proofs being very similar to the proofs for the DN map. Furthermore, in particular the lower bound for the ND map is not as sharp as for the DN map.

\section{M\"{o}bius transformation of the Dirichlet-to-Neumann map} \label{sec:mobdn}

In this section we will relate the DN map of a non-concentric ball inclusion to a DN map for a concentric ball inclusion by the use of M\"obius transformations. This relation will in Section~\ref{sec:dnbnds} be used to obtain bounds on the distinguishability. 

\subsection{M\"obius transformations in the unit disk}

M\"obius transformations are known to preserve harmonic functions in 2D, which makes them \emph{harmonic mor\-phisms}. On the unit disk $\rum{D}$ the harmonic morphisms are uniquely (up to rotation) given by
\begin{equation}
	M_a(x) = \frac{x-a}{\overline{a}x-1},\enskip x\in\rum{D}, \label{mobtrans}
\end{equation}
for $\abs{a}<1$ \cite{Stein2003}. The transformations in \eqref{mobtrans} are special cases of M\"{o}bius transformations, where $M_a : \rum{D}\to\rum{D}$ and $\partial\rum{D}\to\partial\rum{D}$. The particular choice of rotation in \eqref{mobtrans} implies that $M_a$ is an involution, i.e.\ $M_a^{-1} = M_a$. Furthermore, for any ball $B_{C,R}\subset \rum{D}$ with centre $C$ and radius $R<1-\abs{C}$ there exists a unique $a\in\rum{D}$ such that $M_a(B_{C,R}) = B_{0,r}$ for some $r\geq R$.

Let $a \equiv \rho e^{i\sang}$ with $0\leq \rho <1$ and $\sang\in\rum{R}$. Then we can straightforwardly relate the M\"{o}bius transformation anywhere in the disk, $M_a$, to the M\"{o}bius transformation along the real line, $M_\rho$, by the following rotations
\begin{equation}
	M_{\rho e^{i\sang}}(x) = e^{i\sang}M_\rho(e^{-i\sang}x). \label{eq:mobreal}
\end{equation}
This is a useful property that often reduces proofs including $M_a$ to the simpler form $M_\rho$. 

The characterization below of how $M_a$ can be used to move ball inclusions in $\rum{D}$ while preserving harmonic functions is well-known (cf.\ \cite{Hanke_2008,winkler2014}). The proof is short and given for completeness for the particular choice of transformation in \eqref{mobtrans}. 
\begin{proposition} \label{prop:mobdisk}
	\begin{romannum}
		\item Let $a \equiv \rho e^{i\sang}$ with $0\leq \rho < 1$ and $\sang\in\rum{R}$, and let $0<r<1$. Then $M_a$ maps $B_{0,r}$ to $B_{C,R}$ with
		\begin{equation*}
			C = \frac{\rho(r^2-1)}{\rho^2r^2-1}e^{i\sang}, \qquad
			R = \frac{r(\rho^2-1)}{\rho^2r^2-1}. 
		\end{equation*}
		\item Let $C \equiv ce^{i\sang}$ with $0\leq c<1$ and $\sang\in\rum{R}$, and let $0<R < 1-c$. Then the unique $a\in\rum{D}$ such that $M_a$ maps $B_{C,R}$ to a concentric ball $B_{0,r}$ satisfies
		\begin{equation}
			r = \frac{1+R^2-c^2-\sqrt{((1-R)^2-c^2)((1+R)^2-c^2)}}{2R},\qquad
			a = \frac{C}{1-Rr}. \label{eq:ra}
		\end{equation}
	\end{romannum}
\end{proposition}
\begin{proof}
	For (i) we first consider the case $\sang=0$ so $a = \rho$. From \eqref{mobtrans} it is seen that $M_\rho$ is symmetric about the real axis so the centre of $M_\rho(B_{0,r})$ lies on the real axis. Furthermore, the mapping of $M_\rho(r)$ and $M_\rho(-r)$ gives the following real points on $\partial M_\rho(B_{0,r})$:
	\begin{equation*}
		M_\rho(r) = \frac{r-\rho}{\rho r - 1}, \qquad M_\rho(-r) = \frac{r+\rho}{\rho r+1},
	\end{equation*}
	where $M_\rho(-r) > M_\rho(r)$ for all $\rho<1$. Thus centre $c$ and radius $R$ of $M_\rho(B_{0,r})$ can be found as
	\begin{align}
		c &= \frac{M_\rho(-r)+M_\rho(r)}{2} = \frac{\rho(r^2-1)}{\rho^2r^2-1}, \label{ceq} \\
		R &= M_\rho(-r)-c = \frac{r(\rho^2-1)}{\rho^2r^2-1}. \label{Req}
	\end{align}
	Now in the case $\sang\neq 0$ we note that $M_a(B_{0,r}) = e^{i\sang}M_\rho(B_{0,r})$ due to \eqref{eq:mobreal} and that $B_{0,r}$ is rotationally symmetric. So $C = ce^{i\sang}$ which yields the desired result.
	
	For (ii) we solve \eqref{ceq} and \eqref{Req} with respect to $r$ and $\rho$, which for $h \equiv 1+R^2-c^2$ gives
	\begin{equation*}
		r = \frac{h-\sqrt{h^2-4R^2}}{2R}, \qquad \rho = \frac{c}{1-Rr}. 
	\end{equation*}
	By using that $a = \rho e^{i\sang}$ and expanding the terms in $r$ gives the expressions in \eqref{eq:ra}.
\end{proof}

Note from Proposition \ref{prop:mobdisk} that $M_a$ maps the origin $\mathcal{O}$ to $a$ in the same direction as $C$, but a little further towards the boundary as illustrated in Figure \ref{fig:fig2}. However, we will always have that $a\in B_{C,R}$ since $c < 1-R$ and $r<1$ implies
\begin{equation*}
	\abs{a-C} = \rho-c = \frac{c}{1-Rr}-c = \frac{cr}{1-Rr}R < \frac{(1-R)r}{1-Rr}R \leq R. 
\end{equation*}
Thus there is in \eqref{eq:ra} the asymptotic limit
\begin{equation*}
	\lim_{r\to 0} a = \lim_{R\to 0} a = C.
\end{equation*}

\begin{figure}[htb]
\centering
\includegraphics[width = .9\textwidth]{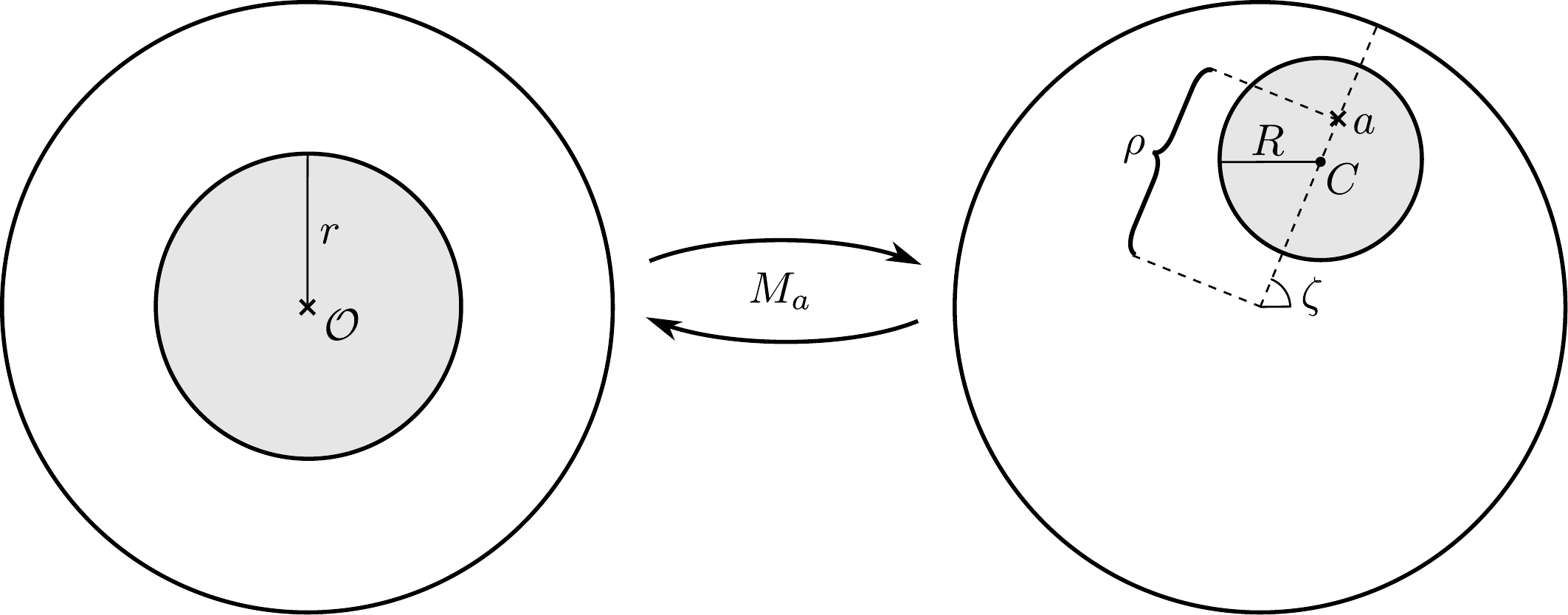}
\caption{Illustration of the action of $M_a$ on ball inclusions in the unit disk $\rum{D}$ using the notation in Proposition \ref{prop:mobdisk}.} \label{fig:fig2}
\end{figure}

Writing $M_a(x) = V_1(x) + iV_2(x)$ for real valued $V_1$ and $V_2$, and similarly ${x = x_1}+ix_2$, then as $M_a$ is holomorphic on $\rum{D}$ the Cauchy-Riemann equations hold
\begin{equation*}
	\tfrac{\partial}{\partial x_1}V_{1} = \tfrac{\partial}{\partial x_2}V_{2}, \quad \tfrac{\partial}{\partial x_2}V_{1} = -\tfrac{\partial}{\partial x_1}V_{2},
\end{equation*}
so the Jacobian determinant of $M_a$ becomes:
\begin{equation}
	J_a(x) = \left(\tfrac{\partial}{\partial x_1}V_{1}\right)^2 + \left(\tfrac{\partial}{\partial x_1}V_{2}\right)^2 = \abs{\tfrac{\partial}{\partial x_1}M_{a}}^2 = \left( \frac{1-\abs{a}^2}{\abs{\overline{a}x - 1}^2} \right)^2. \label{eq:jacdom}
\end{equation}
$J_a$ is the Jacobian determinant for the transformation on the whole domain $\rum{D}$, but for the purpose of transforming the boundary operator $\Lambda(\gamma)$ it is necessary to determine the corresponding transformation on the boundary, i.e.\ determining the tangential and normal part to the Jacobian matrix on $\partial\rum{D}$. Denote for $x\in\rum{D}$ the polar coordinates $x = \beta e^{i\theta}$ and $M_a(x) = Be^{i\Theta}$. We have the following relations on $\partial\rum{D}$:
\begin{align}
	\frac{\partial B}{\partial\theta}|_{\partial\rum{D}} &= \frac{\partial\Theta}{\partial \beta}|_{\partial\rum{D}} = 0, \notag \\
	\frac{\partial B}{\partial \beta}|_{\partial\rum{D}} &= \frac{\partial\Theta}{\partial\theta}|_{\partial\rum{D}} = J_a^{1/2}|_{\partial\rum{D}} = \frac{1-\rho^2}{1+\rho^2-2\rho\cos(\theta-\sang)}. \label{eq:jacdetbdry}
\end{align}
Deriving the terms in \eqref{eq:jacdetbdry} involves straightforward computations using that $M_a$ maps $\partial\rum{D}$ to itself, along with the following identities which are a consequence of the Cauchy-Riemann equations and \eqref{eq:mobreal}
\begin{align*}
		\redel(\tfrac{\partial M_a}{\partial \beta})M_a - \tfrac{\partial M_a}{\partial \beta}\redel(M_a) &= 0,\enskip \text{ on } \partial\rum{D}, \\
		\imdel(\tfrac{\partial M_a}{\partial \beta})M_a - \tfrac{\partial M_a}{\partial \beta}\imdel(M_a) &= 0,\enskip \text{ on } \partial\rum{D}. 
\end{align*}

\subsection{Transformation of the DN map}

In this section we will write up the DN map for the problem transformed by $M_a$ for disk perturbations. Denote $\gamma_{C,R} \equiv 1 + A\chi_{B_{C,R}}$ for $A>-1$ where $\chi_{B_{C,R}}$ is a characteristic function over the open ball $B_{C,R}$ with centre $C$ and radius $R$. Furthermore, the notation in Proposition \ref{prop:mobdisk} will be used throughout, relating $a$ and $r$ to $C$ and $R$. The background conductivity of $1$ is merely for ease of presentation, and can easily be changed to another (constant) background using the identity
\begin{equation*}
	\Lambda(c\gamma) = c\Lambda(\gamma),\enskip c>0.
\end{equation*}

By $\mathcal{M}_a$ we denote the operator applying the transformation $\mathcal{M}_a f \equiv f\circ M_a$, where either $f:\rum{D}\to\rum{C}$ or $f:\partial\rum{D}\to\rum{C}$. Furthermore, we will use the notation $J_a^{1/2}$ both for the square root of \eqref{eq:jacdom} and for the multiplication operator $f\mapsto J_a^{1/2}|_{\partial\rum{D}}f$, indiscriminately. Before investigating the DN map we list a few basic properties.
\begin{proposition} \label{prop:list}
	\begin{romannum}
		\item $\mathcal{M}_a(H^{1/2}(\partial\rum{D})) = H^{1/2}(\partial\rum{D})$ and $\mathcal{M}_a(L^2(\partial\rum{D})) = L^2(\partial\rum{D})$. \label{prop1}
		\item $\mathcal{M}_a$ and $J_a^{1/2}\mathcal{M}_a$ are involutions, i.e.\ their own inverse.\label{prop2}
		\item $J_a^{1/2}\mathcal{M}_a = \mathcal{M}_aJ_a^{-1/2}$ and $J_a^{-1/2}\mathcal{M}_a = \mathcal{M}_aJ_a^{1/2}$.\label{prop3}
		\item ${\mathcal{M}_a}^* = J_a^{1/2}\mathcal{M}_a$ in $L^2(\partial\rum{D})$.\label{prop4}
		\item $J_a^{1/2}$ is on $\partial\rum{D}$ bounded from below and above by positive constants:\label{prop5}
		\begin{equation*}
			\frac{1-\rho}{1+\rho}\leq J_a^{1/2} \leq \frac{1+\rho}{1-\rho}.
		\end{equation*}
	\end{romannum}
\end{proposition}
\begin{proof}
	(iii) is a consequence of the inverse function theorem. For (ii) $\mathcal{M}_a$ is an involution since $M_a$ is an involution, and from (iii)
	\begin{equation*}
		J_a^{1/2}\mathcal{M}_aJ_a^{1/2}\mathcal{M}_a = \mathcal{M}_aJ_a^{-1/2}J_a^{1/2}\mathcal{M}_a = \ident.
	\end{equation*}
	(iv) follows since $M_a^{-1}=M_a$ and $J_a^{1/2}$ is real-valued and is the Jacobian determinant for the boundary integral. For (v) we have
	\begin{align*}
		\inf_{\theta\in(-\pi,\pi)}J_a^{1/2}(e^{i\theta}) &= \inf_{\theta\in(-\pi,\pi)}\frac{1-\rho^2}{\abs{1-\rho e^{i(\theta-\sang)}}^2} = \frac{1-\rho^2}{(1+\rho)^2} = \frac{1-\rho}{1+\rho}, \\
		\sup_{\theta\in(-\pi,\pi)}J_a^{1/2}(e^{i\theta}) &= \sup_{\theta\in(-\pi,\pi)}\frac{1-\rho^2}{\abs{1-\rho e^{i(\theta-\sang)}}^2} = \frac{1-\rho^2}{(1-\rho)^2} = \frac{1+\rho}{1-\rho}.
	\end{align*}
	That $M_a$ is smooth and $J_a^{1/2}$ bounded from below and above by positive constants implies that $\mathcal{M}_a(H^{1/2}(\partial\rum{D}))\subseteq H^{1/2}(\partial\rum{D})$, and $M_a$ being an involution implies the opposite inclusion $H^{1/2}(\partial\rum{D})\subseteq \mathcal{M}_a(H^{1/2}(\partial\rum{D}))$. The same argument is used to show that $\mta(L^2(\partial\rum{D})) = L^2(\partial\rum{D})$.
\end{proof}

Applying $\mta$ to a distribution in $H^{-1/2}(\partial\rum{D})$ is done as a generalization of the change of variables through the dual pairing
\begin{equation*}
	\inner{\mta g,f} \equiv \inner{g,J_a^{1/2}\mta f},\enskip g\in H^{-1/2}(\partial\rum{D}), f\in H^{1/2}(\partial\rum{D}).
\end{equation*}
Now we can write up the DN maps for an inclusion transformed with $\mathcal{M}_a$.
\begin{lemma} \label{dntrans}
	There is the following relation between the DN map for the concentric problem and the DN map for the non-concentric problem:
	\begin{equation}
		\Lambda(\gamma_{C,R}) = \Lambda(\mathcal{M}_a(\gamma_{0,r})) = J_a^{1/2}\mathcal{M}_a\Lambda(\gamma_{0,r})\mathcal{M}_a, \label{dnmaptransform}
	\end{equation}
	and similarly
	\begin{equation*}
		\Lambda(\gamma_{0,r}) = \Lambda(\mathcal{M}_a(\gamma_{C,R})) = J_a^{1/2}\mathcal{M}_a\Lambda(\gamma_{C,R})\mathcal{M}_a.
	\end{equation*}
\end{lemma}
\begin{proof}
	For brevity let $\tilde{w}$ be a shorthand notation for $\mathcal{M}_a w$ where $w$ is either a function on $\partial\rum{D}$ or on $\rum{D}$. Let $u$ be the solution to \eqref{eq:condeq} with conductivity $\gamma_{0,r}$ and Dirichlet condition $u|_{\partial\rum{D}} = f$. Denote the corresponding Neumann condition $g\equiv \Lambda(\gamma_{0,r})f = \nu\cdot\nabla u|_{\partial\rum{D}}$. Furthermore, let $u_1 \equiv u$ in $B_{0,r}$ and $u_2\equiv u$ in $\rum{D}\setminus\overline{B_{0,r}}$. Then as $\gamma_{0,r} = 1+A\chi_{B_{0,r}}$ and $\gamma_{C,R} = 1+A\chi_{B_{C,R}}$ we can write up \eqref{eq:condeq}, along with Dirichlet- and Neumann-conditions as the following system, alongside with the corresponding transformed problem. This gives the following two transmission problems:
	
	\noindent
	\begin{tabular}[!H]{l|l}
	\parbox{0.45\linewidth}{
	\begin{align*}
		\Delta u_1 &= 0 \text{ in } B_{0,r} \\
		\Delta u_2 &= 0 \text{ in } \rum{D}\setminus\overline{B_{0,r}} \\
		u_1 &= u_2 \text{ on } \partial B_{0,r} \\
		(1+A)\eta\cdot\nabla u_1 &= \eta\cdot \nabla u_2 \text{ on } \partial B_{0,r} \\
		u_2 &= f \text{ on } \partial\rum{D} \\
		\nu\cdot\nabla u_2 &= g \text{ on } \partial\rum{D}
	\end{align*}
	} 
	& \parbox{0.45\linewidth}{
	\begin{align*}
		\Delta \tilde{u}_1 &= 0 \text{ in } B_{C,R} \\
		\Delta \tilde{u}_2 &= 0 \text{ in } \rum{D}\setminus\overline{B_{C,R}} \\
		\tilde{u}_1 &= \tilde{u}_2 \text{ on } \partial B_{C,R} \\
		(1+A)\eta\cdot\nabla \tilde{u}_1 &= \eta\cdot \nabla \tilde{u}_2 \text{ on } \partial B_{C,R} \\
		\tilde{u}_2 &= \tilde{f} \text{ on } \partial\rum{D} \\
		\nu\cdot\nabla \tilde{u}_2 &= J_a^{1/2}\tilde{g} \text{ on } \partial\rum{D}
	\end{align*}
	}
	\end{tabular}
	
	Some notational abuse was used as $\eta$ is both unit normal to $\partial B_{0,r}$ and to $\partial B_{C,R}$ in the transformed problem. The Laplace-Beltrami operator is preserved as $M_a$ is a harmonic morphism, and the Dirichlet conditions simply apply the change of variable. The only real change occurs in the derivatives, which on the boundary $\partial B_{C,R}$ cancels out  as $J_a$ is non-zero, and on the outer boundary $\partial\rum{D}$ gives $J_a^{-1/2}\nu\cdot \nabla \tilde{u}_2$ from \eqref{eq:jacdetbdry} and the property $\mathcal{M}_a J_a^{1/2} = J_a^{-1/2}\mathcal{M}_a$.
	
	Thus we have
	\begin{equation*}
		\Lambda(\gamma_{C,R})\tilde{f} = J_a^{1/2}\mathcal{M}_a g = J_a^{1/2}\mathcal{M}_a \Lambda(\gamma_{0,r})f = J_a^{1/2}\mathcal{M}_a \Lambda(\gamma_{0,r})\mathcal{M}_a\tilde{f},\enskip \forall \tilde{f}\in H^{1/2}(\partial\rum{D}).
	\end{equation*}
	One can interchange $\gamma_{0,r}$ and $\gamma_{C,R}$ above by Proposition \ref{prop:list} since $\mathcal{M}_a$ and $J_a^{1/2}\mathcal{M}_a$ are involutions.
\end{proof}

\section{Depth dependent bounds on distinguishability of inclusions} \label{sec:dnbnds}

In this section we determine lower and upper bounds for the distinguishability of $\Lambda(\gamma_{C,R})-\Lambda(1)$, in terms of its largest eigenvalue. The bounds are given in Theorem \ref{dnbounds}. 

The spectrum of $\Lambda(\gamma_{0,r})$ is given below and is derived from a straightforward application of separation of variables, cf.\ \cite[chapter 12.5.1]{Mueller_2012}. Since the eigenfunctions of $\Lambda(\gamma_{0,r})$ and $\Lambda(1)$ are identical, the eigenvalues of the difference operator $\Lambda(\gamma_{0,r}) - \Lambda(1)$ is just the difference of the eigenvalues for the two respective operators. This simplification of course only holds if the eigenfunctions are identical, i.e.\ it will not be the case for the non-concentric problem.
\begin{proposition} \label{prop:eigval}
	For $\gamma_{0,r} \equiv 1+A\chi_{B_{0,r}}$ with $0<r<1$ and $A>-1$, the eigenfunctions of $\Lambda(\gamma_{0,r})$ are $f_n(\theta) \equiv \frac{1}{\sqrt{2\pi}}e^{in\theta},\enskip n\in\rum{Z}$. The corresponding eigenvalues are
	\begin{equation*}
		\lambda_n \equiv \frac{2+A(1+r^{2\abs{n}})}{2+A(1-r^{2\abs{n}})}\abs{n},\enskip n\in\rum{Z}.
	\end{equation*}
	The eigenvalues for the difference operator $\Lambda(\gamma_{0,r})-\Lambda(1)$ are
	\begin{equation}
		\lambda_n \equiv \frac{2Ar^{2\abs{n}}}{2+A(1-r^{2\abs{n}})}\abs{n},\enskip n\in\rum{Z}. \label{eq:lamdiff}
	\end{equation}
\end{proposition}

\begin{remark}
	The eigenvalues in \eqref{eq:lamdiff} are not necessarily monotonously decaying in $\abs{n}$. This depends on the values of $A$ and $r$. This is unlike the Neumann-to-Dirichlet operators for which the eigenvalues have monotonous decay as seen in Proposition~\ref{lemndeig}.
\end{remark}

$\Lambda(\gamma)$ is an unbounded operator on $L^2(\partial\rum{D})$ for any $\gamma\in L^\infty_+(\rum{D})$, however the difference $\Lambda(\gamma_{C,R})-\Lambda(1)$ is infinitely smoothing as $\gamma_{C,R} = 1$ in a neighbourhood of $\partial\rum{D}$ (see e.g.\ \cite[Lemma 3.1]{Cornean2006}). In fact $\Lambda(\gamma_{C,R})-\Lambda(1)$ extends continuously to a compact and self-adjoint operator on all of $L^2(\partial\rum{D})$, and it is for this extension that we determine distinguishability bounds. In lack of a proper reference to such a result we give the proof below for our specific scenario.
\begin{lemma} \label{lemma:DNL2bnd}
	For each centre $C$ and radius $R$ such that $\overline{B_{C,R}}\subset\rum{D}$, the operator $\Lambda(\gamma_{C,R})-\Lambda(1)$ continuously extends to a compact and self-adjoint operator in $\mathcal{L}(L^2(\partial\rum{D}))$.
\end{lemma}
\begin{proof}
The eigenfunctions in Proposition~\ref{prop:eigval} comprise the orthonormal Fourier basis $\{f_n\}_{n\in\rum{Z}}$ for $L^2(\partial\rum{D})$. Using that $\Lambda(\gamma_{0,r})$ and $\Lambda(1)$ are symmetric operators w.r.t.\ the $L^2(\partial\rum{D})$-inner product, implies that the difference operator $\Lambda(\gamma_{0,r})-\Lambda(1)$ can be written as below, where $\lambda_n$ denotes the eigenvalues in \eqref{eq:lamdiff}:
\begin{equation}
	(\Lambda(\gamma_{0,r})-\Lambda(1))f = \sum_{n\in\rum{Z}} \lambda_n\inner{f,f_n}f_n, \enskip f\in H^{1/2}(\partial\rum{D}). \label{eq:lamdiff_exp}
\end{equation}
Since $\sup_{n\in\rum{Z}}\abs{\lambda_n} < \infty$ then \eqref{eq:lamdiff_exp} implies that $\Lambda(\gamma_{0,r})-\Lambda(1)$ is bounded in terms of the $L^2(\partial\rum{D})$-norm:
\begin{equation*}
	\norm{(\Lambda(\gamma_{0,r})-\Lambda(1))f}_{L^2(\partial\rum{D})} \leq \sup_{n\in\rum{Z}}\abs{\lambda_n}\norm{f}_{L^2(\partial\rum{D})}, \enskip f\in H^{1/2}(\partial\rum{D}),
\end{equation*}
i.e.\ using the formula in \eqref{eq:lamdiff_exp} the operator $\Lambda(\gamma_{0,r})-\Lambda(1)$ continuously extends to a self-adjoint operator in $\mathcal{L}(L^2(\partial\rum{D}))$. 

Note that $\abs{\lambda_n}\leq 2\abs{A}\abs{n}r^{2\abs{n}} \to 0$ for $n\to\infty$ implies that the extension is compact. This follows as $\Lambda(\gamma_{0,r})-\Lambda(1)$ is the limit of the finite rank operators $P_N(\Lambda(\gamma_{0,r})-\Lambda(1))$, where $P_N$ is the orthogonal projection onto $\mspan\{f_n\}_{\abs{n}\leq N}$,
\begin{align*}
	\norm{(P_N-\ident)(\Lambda(\gamma_{0,r})-\Lambda(1))}_{\mathcal{L}(L^2(\partial\rum{D}))}^2 &= \sup_{f\in L^2(\partial\rum{D})\setminus\{0\}} \frac{1}{\norm{f}^2}\sum_{\abs{n}>N} \abs{\lambda_n}^2\abs{\inner{f,f_n}}^2 \\
	&\leq \sup_{\abs{n}>N}\abs{\lambda_n}^2 \to 0 \text{ for } N\to\infty.
\end{align*}
Since $\mta$ and $J_a^{1/2}\mta$ belong to $\mathcal{L}(L^2(\partial\rum{D}))$ implies that through \eqref{dnmaptransform} then $\Lambda(\gamma_{C,R})-\Lambda(1)$ extends to a compact and self-adjoint operator in $\mathcal{L}(L^2(\partial\rum{D}))$, for any centre $C$ and radius $R$.
\end{proof}

For brevity we will denote by $\norm{\cdot}$ the operator norm on $\mathcal{L}(L^2(\partial\rum{D}))$, and it should be straightforward to distinguish it from the $L^2(\partial\rum{D})$-norm from the context it is used. It is well known from the spectral theorem that the operator norm of a compact and self-adjoint Hilbert space operator equals the largest magnitude eigenvalue of the operator, and is furthermore given by
\begin{align}
	\norm{\Lambda(\gamma_{C,R})-\Lambda(1)} &= \sup_{f\in L^2(\partial\rum{D})\setminus\{0\}}\frac{\norm{(\Lambda(\gamma_{C,R})-\Lambda(1))f}}{\norm{f}} \notag \\
	&= \sup_{f\in L^2(\partial\rum{D})\setminus\{0\}}\frac{\abs{\inner{(\Lambda(\gamma_{C,R})-\Lambda(1))f,f}}}{\norm{f}^2}. \label{eq:maxeiginnerCR}
\end{align}
Thus in reality the distinguishability is related to a choice of boundary condition (here Dirichlet condition). Choosing the eigenfunction $f_1$ to the largest magnitude eigenvalue $\lambda_1$ of $\Lambda(\gamma_{C,R})-\Lambda(1)$ maximizes the expression in \eqref{eq:maxeiginnerCR}. The min-max theorem (see e.g.\ \cite{MathPhysVol4}) furthermore states that in the orthogonal complement to $f_1$, the maximizing function is $f_2$, the eigenfunction to the second largest eigenvalue $\lambda_2$. Continuing the procedure gives an orthonormal set of boundary conditions that in each orthogonal direction maximizes the difference $(\Lambda(\gamma_{C,R})-\Lambda(1))f$. 

Suppose that we instead have a noisy approximation $\Lambda^\delta \equiv \Lambda(\gamma_{C,R}) + E^\delta$ with noise level $\norm{E^\delta} = \delta$. If we hope to be able to recover the inclusion $B_{C,R}$ from $\Lambda^\delta$ then we need $\norm{\Lambda(\gamma_{C,R})-\Lambda(1)} > \delta$ in order to \emph{distinguish} that the data $\Lambda^\delta$ does not come from the background conductivity $\gamma = 1$, and that there is an inclusion to reconstruct. The distinguishability is therefore a measure of how much noise that can be added before the structural information is completely lost. In particular the magnitude of the eigenvalues for $\Lambda(\gamma_{C,R})-\Lambda(1)$ shows whether the corresponding eigenfunctions are able to contribute any distinguishability for a given noise level.

Even though the eigenvalues for the concentric problem are known, this does not imply that Lemma \ref{dntrans} directly gives the spectrum of the non-concentric problem. As seen below, an eigenfunction $f$ of $\Lambda(\gamma_{0,r})$ does not yield an eigenfunction $\mathcal{M}_a f$ of $\Lambda(\gamma_{C,R})$ but is instead an eigenfunction of the operator scaled by $J_a^{-1/2}$. 
\begin{corollary}
	$(\lambda,f)$ is an eigenpair of $\Lambda(\gamma_{0,r})$ if and only if $(\lambda,\mathcal{M}_a f)$ is an eigenpair of $J_a^{-1/2}\Lambda(\gamma_{C,R})$.
\end{corollary}
\begin{proof}
	From Lemma \ref{dntrans} we have:
	\begin{equation}
		J_a^{-1/2}\Lambda(\gamma_{C,R})\mathcal{M}_a f = \mathcal{M}_a\Lambda(\gamma_{0,r}) f. \label{eq:coro1}
	\end{equation}
	If $(\lambda,f)$ is an eigenpair of $\Lambda(\gamma_{0,r})$ then \eqref{eq:coro1} gives $J_a^{-1/2}\Lambda(\gamma_{C,R})\mathcal{M}_a f = \lambda\mathcal{M}_a f$. On the other hand, if $(\lambda,\mathcal{M}_a f)$ is an eigenpair of $J_a^{-1/2}\Lambda(\gamma_{C,R})$ then \eqref{eq:coro1} gives $\mathcal{M}_a\Lambda(\gamma_{0,r})f = \lambda\mathcal{M}_a f$ and as $\mathcal{M}_a^{-1} = \mathcal{M}_a$ then $(\lambda,f)$ is an eigenpair of $\Lambda(\gamma_{0,r})$.
\end{proof}

To the authors' knowledge there is not a known closed-form expression for either eigenvalues or eigenfunctions of the non-concentric problem. However, it is possible to obtain explicit bounds, and for these bounds we will make use of certain weighted norms. 

Since $J_a^{1/2}$ is real-valued and bounded as in Proposition \ref{prop:list} gives rise to other weighted norms and inner products on $L^2(\partial\rum{D})$, namely
\begin{alignat}{2}
	\inner{f,g}_{1/2} &\equiv \int_{\partial\rum{D}} f\overline{g}J_a^{1/2}\,ds, &\quad \norm{f}_{1/2} &\equiv \sqrt{\inner{f,f}_{1/2}}, \label{eq:halfnorm} \\
	\inner{f,g}_{-1/2} &\equiv \int_{\partial\rum{D}} f\overline{g}J_a^{-1/2}\,ds, &\quad \norm{f}_{-1/2} &\equiv \sqrt{\inner{f,f}_{-1/2}}. \label{eq:mhalfnorm}
\end{alignat}
It is clear from Proposition \ref{prop:list}(v) that these weighted norms are equivalent to the usual $L^2(\partial\rum{D})$-norm:
\begin{equation}
	\sqrt{\frac{1-\rho}{1+\rho}}\norm{f} \leq \norm{f}_{\pm 1/2} \leq \sqrt{\frac{1+\rho}{1-\rho}} \norm{f}, \enskip f\in L^2(\partial\rum{D}). \label{eq:normeqv}
\end{equation}
The weighted norms are used below in Theorem \ref{dnbounds} for determining bounds on the distinguishability. The weighted inner products will turn out to be a natural choice when determining an exact matrix representation for $\Lambda(\gamma_{C,R})-\Lambda(1)$, as seen in \Sref{sec:bnds}.
\begin{theorem} \label{dnbounds}
	Let $\gamma$ be either $\gamma_{0,r}$ or $\gamma_{C,R}$. From the weighted norms \eqref{eq:halfnorm} and \eqref{eq:mhalfnorm} we obtain
	\begin{equation}
		\norm{\Lambda(\gamma)-\Lambda(1)} = \sup_{f\in L^2(\partial\rum{D})\setminus\{0\} }\frac{\norm{(\Lambda(\mathcal{M}_a\gamma)-\Lambda(1))f}_{-1/2}}{\norm{f}_{1/2}}. \label{dnnorms}
	\end{equation}
	Furthermore the following bounds hold 
	\begin{equation}
		\frac{1-\rho}{1+\rho}\norm{\Lambda(\gamma_{C,R}) - \Lambda(1)} \leq \norm{\Lambda(\gamma_{0,r}) - \Lambda(1)} \leq \sqrt{\frac{1-\rho^2}{1+\rho^2}}\norm{\Lambda(\gamma_{C,R}) - \Lambda(1)}. \label{eq:dnbnds}
	\end{equation}
\end{theorem}
\begin{proof}
	By Lemma \ref{dntrans}
	\begin{align*}
		\norm{\Lambda(\gamma)-\Lambda(1)}^2 &= \norm{J_a^{1/2}\mathcal{M}_a(\Lambda(\mathcal{M}_a\gamma)-\Lambda(1))\mathcal{M}_a}^2 \\
		&= \sup_{f\in L^2(\partial\rum{D})\setminus\{0\} }\frac{\norm{J_a^{1/2}\mathcal{M}_a(\Lambda(\mathcal{M}_a\gamma)-\Lambda(1))\mathcal{M}_a f}^2}{\norm{f}^2}.
	\end{align*}
	Now applying the change of variables with $\mta$ in both numerator and denominator, and using that $J_a^{1/2}$ is the Jacobian determinant in the boundary integral along with Proposition \ref{prop:list}(iii), yields
	\begin{align*}
		\norm{\Lambda(\gamma)-\Lambda(1)}^2 &= \sup_{f\in L^2(\partial\rum{D})\setminus\{0\} }\frac{\int_{\partial\rum{D}} J_a\mathcal{M}_a\abs{(\Lambda(\mathcal{M}_a\gamma)-\Lambda(1))\mathcal{M}_a f}^2\,ds }{\int_{\partial\rum{D}}\abs{f}^2\,ds} \\
		&= \sup_{f\in L^2(\partial\rum{D})\setminus\{0\} }\frac{\int_{\partial\rum{D}} J_a^{-1/2}\abs{(\Lambda(\mathcal{M}_a\gamma)-\Lambda(1))\mathcal{M}_a f}^2\,ds }{\int_{\partial\rum{D}}J_a^{1/2}\abs{\mta f}^2\,ds}.
	\end{align*}
	Finally, it is applied that $\mta f$ can be substituted by $f$ in the supremum since $\mta(L^2(\partial\rum{D})) = L^2(\partial\rum{D})$ and $\mathcal{M}_a f = 0 \Leftrightarrow f = 0$
	\begin{equation*}
		\norm{\Lambda(\gamma)-\Lambda(1)}^2 = \sup_{f\in L^2(\partial\rum{D})\setminus\{0\} }\frac{\int_{\partial\rum{D}} J_a^{-1/2}\abs{(\Lambda(\mathcal{M}_a\gamma)-\Lambda(1)) f}^2\,ds }{\int_{\partial\rum{D}}J_a^{1/2}\abs{f}^2\,ds},
	\end{equation*}
	which is the expression in \eqref{dnnorms}. 
	
	Let $f_1$ be the eigenfunction of $\Lambda(\gamma_{C,R})-\Lambda(1)$ corresponding to the largest eigenvalue $\lambda_1$, and similarly let $\hat{f}_1$ be the eigenfunction of $\Lambda(\gamma_{0,r})-\Lambda(1)$ corresponding to the largest eigenvalue $\hat{\lambda}_1$, then
	\begin{equation*}
		\abs{\hat{\lambda}_1} = \norm{\Lambda(\gamma_{0,r})-\Lambda(1)} \geq \frac{\norm{(\Lambda(\gamma_{C,R})-\Lambda(1))f_1}_{-1/2}}{\norm{f_1}_{1/2}} = \abs{\lambda_1}\frac{\norm{f_1}_{-1/2}}{\norm{f_1}_{1/2}}.
	\end{equation*}
	Now utilizing the norm equivalence in \eqref{eq:normeqv}
	\begin{equation*}
		\abs{\hat{\lambda}_1} \geq \abs{\lambda_1}\frac{\norm{f_1}_{-1/2}}{\norm{f_1}_{1/2}} \geq \abs{\lambda_1}\frac{\sqrt{\frac{1-\rho}{1+\rho}}\norm{f_1}}{\sqrt{\frac{1+\rho}{1-\rho}}\norm{f_1}} = \frac{1-\rho}{1+\rho}\abs{\lambda_1},
	\end{equation*}
	which is the lower bound in \eqref{eq:dnbnds}. The same can be done by interchanging $\gamma_{0,r}$ and $\gamma_{C,R}$
	\begin{equation*}
		\abs{\lambda_1} = \norm{\Lambda(\gamma_{C,R})-\Lambda(1)} \geq \frac{\norm{(\Lambda(\gamma_{0,r})-\Lambda(1))\hat{f}_1}_{-1/2}}{\norm{\hat{f}_1}_{1/2}} = \abs{\hat{\lambda}_1}\frac{\norm{\hat{f}_1}_{-1/2}}{\norm{\hat{f}_1}_{1/2}}.
	\end{equation*}
	Since $\gamma_{0,r}$ is concentric, then $\hat{f}_1$ may be chosen as a complex exponential by Proposition~\ref{prop:eigval} i.e.\ $\abs{\hat{f}_1} = 1$
	\begin{equation}
		\abs{\hat{\lambda}_1} \leq \abs{\lambda_1}\sqrt{\frac{\int_{\partial\rum{D}}J_a^{1/2}\,ds}{\int_{\partial\rum{D}}J_a^{-1/2}\,ds}}. \label{eq:bnd1}
	\end{equation}
	Here $\int_{\partial\rum{D}}J_a^{1/2}\, ds = \int_{\partial\rum{D}} 1\,ds = 2\pi$ as $J_a^{1/2}$ is the Jacobian determinant in the boundary integral. By \eqref{eq:jacdetbdry}
	\begin{equation}
		\int_{\partial\rum{D}} J_a^{-1/2}\,dx = \frac{1}{1-\rho^2}\int_0^{2\pi} \left[1+\rho^2-2\rho\cos(\theta-\sang)\right]d\theta = 2\pi\frac{1+\rho^2}{1-\rho^2}, \label{Jmhalfint}
	\end{equation}
	which combined with \eqref{eq:bnd1} gives the upper bound in \eqref{eq:dnbnds}
	\begin{equation*}
		\abs{\hat{\lambda}_1} \leq \abs{\lambda_1}\sqrt{\frac{1-\rho^2}{1+\rho^2}}.
	\end{equation*}
\end{proof}

In the bounds in Theorem \ref{dnbounds} it is worth noting that both lower and upper bound tend to zero as $\rho$ tends to $1$. When $\rho$ approaches 1, $B_{C,R}$ approaches $\partial\rum{D}$, and the largest eigenvalue of $\Lambda(\gamma_{C,R})-\Lambda(1)$ tends to infinity corresponding to $\Lambda(\gamma_{C,R})-\Lambda(1)$ diverging in $\mathcal{L}(L^2(\partial\rum{D}))$.

Since the constant in the upper bound in \eqref{eq:dnbnds} is smaller than 1 for any $0\leq\rho<1$ implies that $\norm{\Lambda(\gamma_{0,r})-\Lambda(1)} \leq \norm{\Lambda(\gamma_{C,R})-\Lambda(1)}$ for any $a\in\rum{D}$. This means that the distinguishability increases as the inclusion is moved closer to the boundary. However it does so even though $B_{C,R}$ is decreasing in size as $\lim_{\rho\to 1}R = 0$. So no matter what the size of $B_{0,r}$ is, it is always possible to construct another arbitrarily small inclusion $B_{C,R}$ sufficiently close to the boundary $\partial\rum{D}$ such that $\Lambda(\gamma_{C,R})$ is easier to distinguish from $\Lambda(1)$ than $\Lambda(\gamma_{0,r})$ is, in the presence of noise. In other words, given a noisy measurement we can expect to more stably reconstruct smaller structures of $\gamma$ near the boundary than larger structures deeper in the domain.

Combining \eqref{eq:dnbnds} with Corollary \ref{coro:mono} in Appendix \ref{sec:appA} directly gives the following upper bound on the distinguishability when then size of the inclusion is fixed.
\begin{corollary} \label{coro:fixedsize}
	For $\abs{C}\leq 1-r$ the following bounds hold
	\begin{equation*}
		\norm{\Lambda(\gamma_{0,r})-\Lambda(1)} \leq \sqrt{\frac{1-\rho^2}{1+\rho^2}}\norm{\Lambda(\gamma_{C,R})-\Lambda(1)} \leq \sqrt{\frac{1-\rho^2}{1+\rho^2}}\norm{\Lambda(\gamma_{C,r})-\Lambda(1)}.
	\end{equation*}
\end{corollary}

\section{Comparison of bounds on the dinstinguishability} \label{sec:bnds}

In this section the bounds from Theorem \ref{dnbounds} are investigated and verified numerically, to see how tight the bounds are for inclusions of various sizes. Here it is important to determine eigenvalues of the non-concentric problem $\Lambda(\gamma_{C,R})-\Lambda(1)$ accurately. Therefore, we will avoid numerical solution of $f\mapsto \Lambda(\gamma_{C,R})f$ as well as numerical integration, as integration of high frequent trigonometric-like functions requires many sampling points for a usual Gauss-Legendre quadrature rule to be accurate. Instead we will use an orthonormal basis $\{\phi_n\}_{n\in\rum{Z}}$ in terms of the inner product $\inner{\cdot,\cdot}_{1/2}$ from \eqref{eq:halfnorm}, and determine the coefficients
\begin{equation*}
	\mathcal{A}_{n,m} \equiv \inner{(\Lambda(\gamma_{C,R})-\Lambda(1))\phi_m,\phi_n}_{1/2}
\end{equation*}
exactly, based on the known spectrum of the concentric problem $\Lambda(\gamma_{0,r})-\Lambda(1)$ and the transformation $\mathcal{M}_a$ that takes $B_{0,r}$ to $B_{C,R}$. As the basis is orthonormal the infinite dimensional matrix $\mathcal{A}$ is then a matrix representation of $\Lambda(\gamma_{C,R})-\Lambda(1)$. This is understood in the sense that for $f\in H^{1/2}(\partial\rum{D})$ where we write $f = \sum_{m\in\rum{Z}}v_m\phi_m$ with the coefficients $v_m \equiv \inner{f,\phi_m}_{1/2}$ collected in a sequence $v$, then the $n$'th component of $\mathcal{A}v$ is by linearity of $\Lambda(\gamma_{C,R})-\Lambda(1)$ and the inner product given by
\begin{align*}
	(\mathcal{A}v)_n &= \sum_{m\in\rum{Z}}\inner{f,\phi_m}_{1/2}\inner{(\Lambda(\gamma_{C,R})-\Lambda(1))\phi_m,\phi_n}_{1/2} \\
	&= \inner{(\Lambda(\gamma_{C,R})-\Lambda(1))f,\phi_n}_{1/2}.
\end{align*}
Thus $\mathcal{A}$ maps the basis coefficients for $f$ to the corresponding basis coefficients of $(\Lambda(\gamma_{C,R})-\Lambda(1))f$. Furthermore, $\mathcal{A}$ has the same eigenvalues as $\Lambda(\gamma_{C,R})-\Lambda(1)$, and the eigenvectors of $\mathcal{A}$ comprise the basis coefficients for the eigenfunctions of $\Lambda(\gamma_{C,R})-\Lambda(1)$. In practice we can only construct an $N$-term approximation $\mathcal{A}_N$ using the finite $\{\phi_n\}_{\abs{n}\leq N}$ set of basis functions. Such a matrix is a representation of the operator
\[
P_N(\Lambda(\gamma_{C,R})-\Lambda(1))P_N, 
\]
where $P_N$ is an orthogonal projection onto $\mspan\{\phi_n\}_{\abs{n}\leq N}$ in terms of the $\inner{\cdot,\cdot}_{1/2}$-inner product. For compact operators it is known from spectral theory (cf.\ \cite{Osborn1975,Kato1995}) that eigenvalues and eigenfunctions of such $N$-term approximations converge as $N\to\infty$. From Figure \ref{fig:fig3} it is evident that it is possible to estimate the correct eigenvalues to machine precision using very small $N$ if the basis $\{\phi_n\}_{n\in\rum{Z}}$ is well-chosen.

Let $f_n(\theta) \equiv \frac{1}{\sqrt{2\pi}}e^{in\theta}$ be the usual Fourier basis for $L^2(\partial\rum{D})$. Since $\{f_n\}_{n\in\rum{Z}}$ is an orthonormal basis in the usual $L^2(\partial\rum{D})$-inner product, it follows straightforwardly that $\phi_n \equiv \mathcal{M}_a f_n$ gives an orthonormal basis in the $\inner{\cdot,\cdot}_{1/2}$-inner product. It is a consequence of Proposition \ref{prop:list} and that $\mathcal{M}_a$ is bounded; by picking $f\in L^2(\partial\rum{D})$ then $\mathcal{M}_a f\in L^2(\partial\rum{D})$ so
\begin{equation*}
	\mathcal{M}_a f= \sum_{n\in\rum{Z}}\inner{\mathcal{M}_a f,f_n}f_n = \sum_{n\in\rum{Z}} \inner{f,\phi_n}_{1/2}f_n \Rightarrow f = \sum_{n\in\rum{Z}} \inner{f,\phi_n}_{1/2}\phi_n.
\end{equation*}

\begin{theorem} \label{dnmat} 
	Let $\hat{\lambda}_n$ be the $n$'th eigenvalue of $\Lambda(\gamma_{0,r})-\Lambda(1)$ (cf.\ Proposition~\ref{prop:eigval}). Define the orthonormal basis $\{\phi_n\}_{n\in\rum{Z}}$ by
	\begin{equation*}
		\phi_n \equiv \mta f_n, \quad f_n(\theta) \equiv \frac{1}{\sqrt{2\pi}}e^{in\theta}, \enskip n\in\rum{Z}. 
	\end{equation*}
	Then $\Lambda(\gamma_{C,R})-\Lambda(1)$ is represented in this basis via the following tridiagonal matrix:
	\begin{equation*}
		\mathcal{A}_{m,n} \equiv \inner{(\Lambda(\gamma_{C,R})-\Lambda(1))\phi_m,\phi_n}_{1/2} = \begin{cases}
		\tfrac{1+\rho^2}{1-\rho^2}\hat{\lambda}_m, & m=n, \\
		\tfrac{-a}{1-\rho^2}\hat{\lambda}_m, & m-n=1, \\
		\tfrac{-\overline{a}}{1-\rho^2}\hat{\lambda}_m, & m-n=-1, \\
		0, & \textup{else}.
		\end{cases}
	\end{equation*}
\end{theorem}
\begin{proof}
	Utilizing Lemma \ref{dntrans} and Proposition \ref{prop:list} (and that $\Lambda(\mta 1) = \Lambda(1)$):
	\begin{align*}
		\inner{(\Lambda(\gamma_{C,R})-\Lambda(1))\phi_m,\phi_n}_{1/2} &= \inner{J_a^{1/2}\mta (\Lambda(\gamma_{0,r})-\Lambda(1)) \mta\mta f_m,J_a^{1/2}\mta f_n} \\
		&= \inner{(\Lambda(\gamma_{0,r})-\Lambda(1)) f_m,J_a^{-1/2}f_n}.
	\end{align*}
	Now using that $f_m$ is an eigenfunction of $\Lambda(\gamma_{0,r})-\Lambda(1)$ and the expression \eqref{eq:jacdetbdry} for $J_a^{1/2}|_{\partial\rum{D}}$
	\begin{align*}
		\inner{(\Lambda(\gamma_{C,R})-\Lambda(1))\phi_m,\phi_n}_{1/2} &= \frac{1}{2\pi}\inner{(\Lambda(\gamma_{0,r})-\Lambda(1)) e^{im\theta}, J_a^{-1/2}e^{in\theta}} \\
		&= \frac{\hat{\lambda}_m}{2\pi}\inner{e^{im\theta}, J_a^{-1/2}e^{in\theta}} \\
		&= \frac{\hat{\lambda}_m}{2\pi(1-\rho^2)}\int_{0}^{2\pi} e^{i(m-n)\theta}(1+\rho^2-2\rho\cos(\theta-\sang))d\theta \\
		&= \begin{cases}
		\tfrac{1+\rho^2}{1-\rho^2}\hat{\lambda}_m, & m=n, \\
		\tfrac{-a}{1-\rho^2}\hat{\lambda}_m, & m-n=1, \\
		\tfrac{-\overline{a}}{1-\rho^2}\hat{\lambda}_m, & m-n=-1, \\
		0, & \textup{else}.
		\end{cases}
	\end{align*}
	So the above calculation gives the matrix representation.
\end{proof}

\begin{figure}[htb]
\centering
\subfloat[\label{fig:fig3a}]{\includegraphics[width=.49\textwidth]{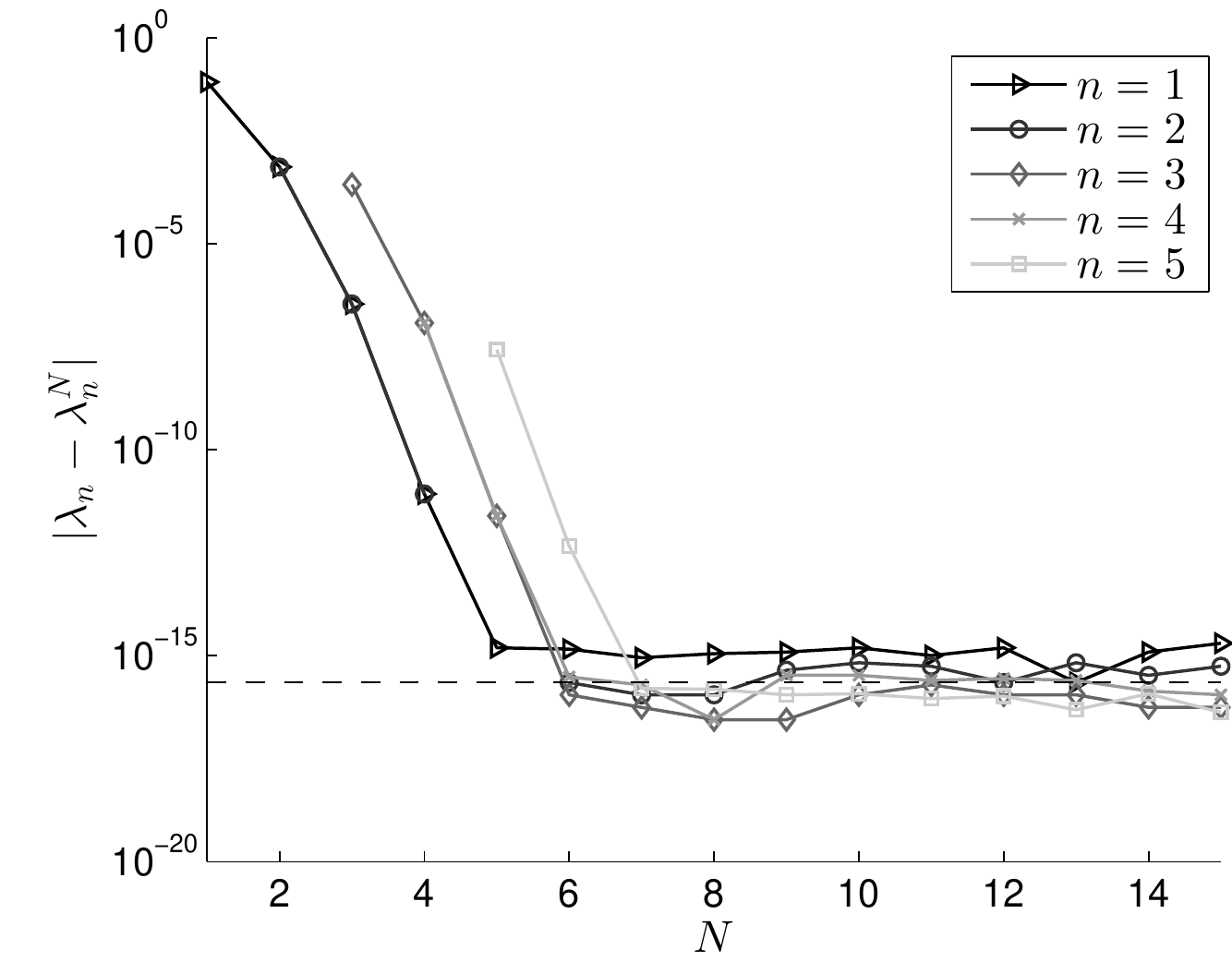}} 
\hfill
\subfloat[\label{fig:fig3b}]{\includegraphics[width=.49\textwidth]{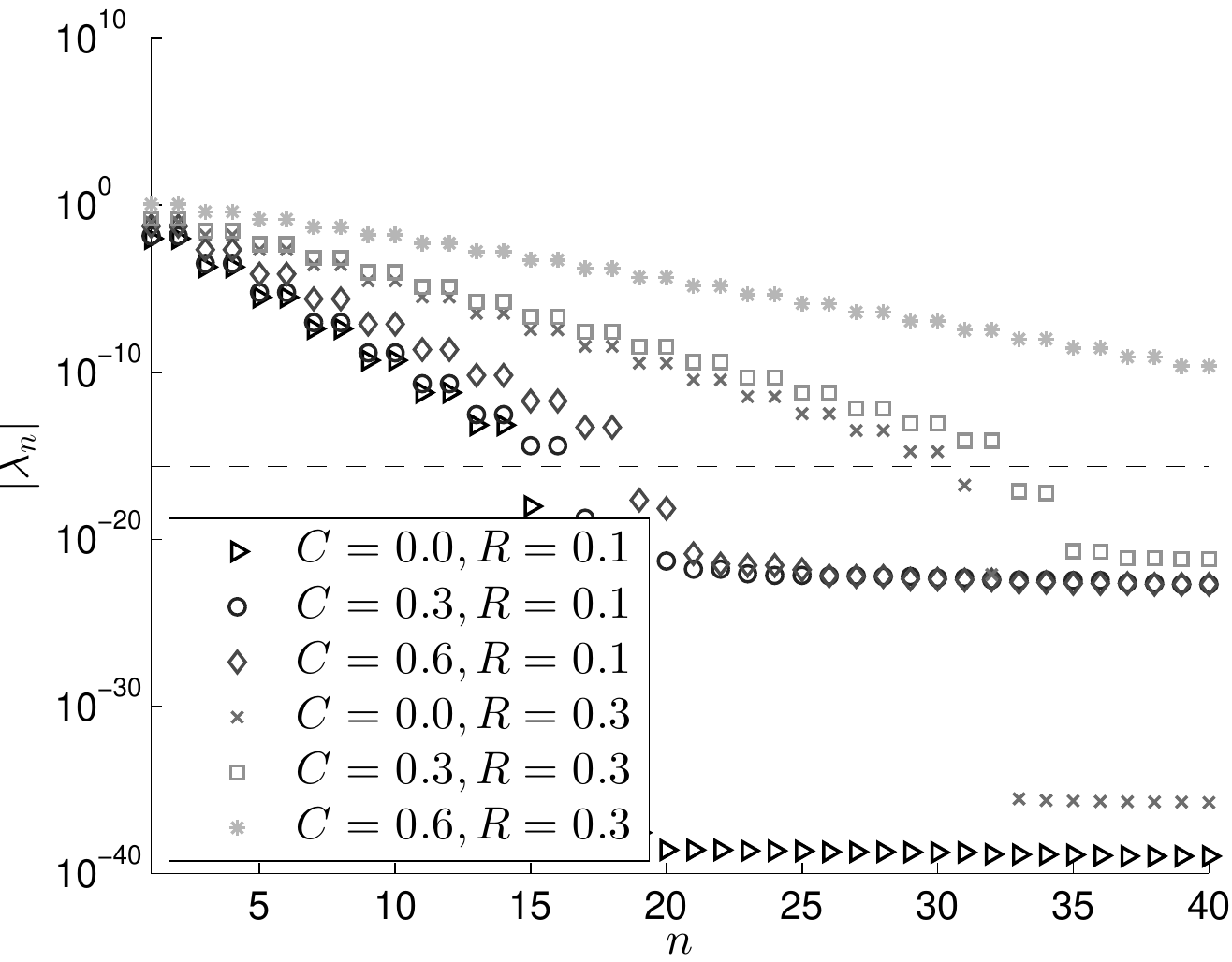}} 
\caption{\textbf{(a):} Difference $\abs{\lambda_n-\lambda_n^N}$ between the $n=1,2,\dots,5$ largest eigenvalues $\lambda_n$ of $\Lambda(\gamma_{C,R})-\Lambda(1)$ with $C = 0.7$ and $R = 0.2$, and the eigenvalues $\lambda_n^N$ of the $N$-term approximation $\mathcal{A}_N$ from Theorem \ref{dnmat}. \textbf{(b):} Largest $n=1,2,\dots,40$ eigenvalues of $\Lambda(\gamma_{C,R})-\Lambda(1)$ for various values of $C$ and $R$, estimated to machine precision (dashed line).}
\label{fig:fig3}
\end{figure}

The basis functions in Theorem \ref{dnmat} can explicitly be given in terms of $\theta$. Since $M_a:\partial\rum{D}\to\partial\rum{D}$ then the angular variable $\theta$ is mapped to another angular variable $\psi_a(\theta)$, thus
\begin{equation*}
	\phi_n(\theta) = \frac{1}{\sqrt{2\pi}}e^{in\psi_a(\theta)} = \frac{1}{\sqrt{2\pi}}M_a (e^{i\theta})^n = \frac{1}{\sqrt{2\pi}}\left(\frac{e^{i\theta}-\rho e^{i\sang}}{\rho e^{i(\theta-\sang)}-1}\right)^n, \enskip n\in\rum{Z}. 
\end{equation*}
\begin{remark}
	The matrix $\mathcal{A}$ in Theorem \ref{dnmat} is not Hermitian as $\Lambda(\gamma_{C,R})-\Lambda(1)$ is only self-adjoint in the regular $L^2(\partial\rum{D})$-inner product, and not in the weighted $\inner{\cdot,\cdot}_{1/2}$-inner product.
\end{remark}

The ratio of the norms in Theorem \ref{dnbounds} have negligible dependence with respect to the amplitude $A$, compared to the radius $r$ (note also that $\rho$ in the bounds are independent of $A$). This can also be seen in terms of the Fr\'echet derivate of $\gamma\mapsto \Lambda(\gamma)$:
\begin{equation*}
	\frac{\norm{\Lambda(1+A\chi_{B_{0,r}})-\Lambda(1)}}{\norm{\Lambda(1+A\chi_{B_{C,R}})-\Lambda(1)}} = \frac{\norm{\Lambda'(1)\chi_{B_{0,r}}+ o(A)/A}}{\norm{\Lambda'(1)\chi_{B_{C,R}}+ o(A)/A}} \xrightarrow[A\to 0]{} \frac{\norm{\Lambda'(1)\chi_{B_{0,r}}}}{\norm{\Lambda'(1)\chi_{B_{C,R}}}}.
\end{equation*}
Therefore $A$ will be kept fixed $A = 2$ in the following examples.

\begin{figure}[htb]
\centering
\includegraphics[width=.6\textwidth]{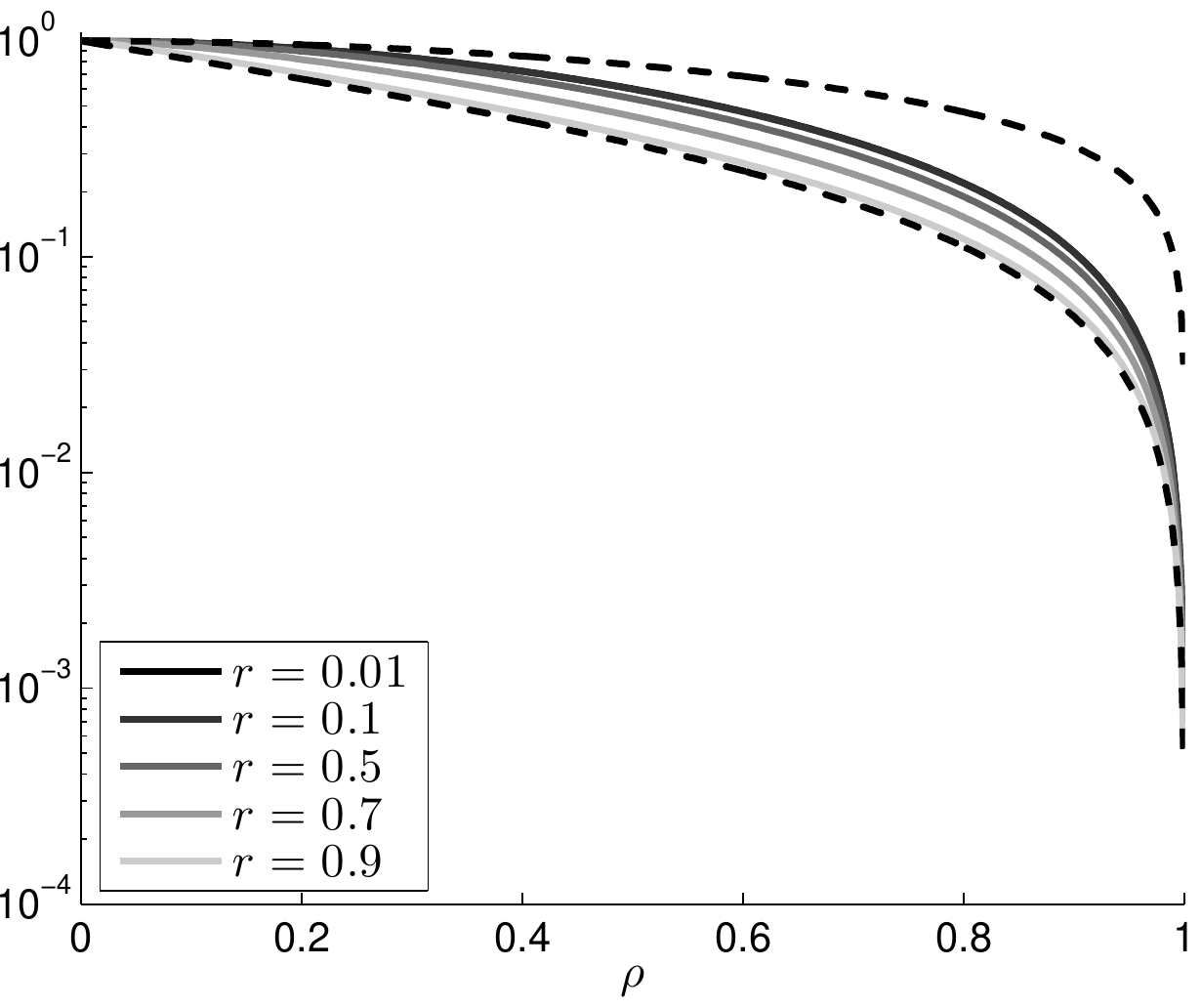}
\caption{Ratio $\norm{\Lambda(\gamma_{0,r})-\Lambda(1)}/\norm{\Lambda(\gamma_{C,R})-\Lambda(1)}$ for $\abs{a} = \rho\in[0,1)$ where $R$ and $C$ are determined from $r$ and $\rho$ by Proposition \ref{prop:mobdisk}, along with the bounds (dashed lines) from Theorem \ref{dnbounds}.}
\label{fig:fig4}
\end{figure}

\begin{figure}[htb]
\centering
\subfloat[\label{fig:fig5a}]{\includegraphics[width=.49\textwidth]{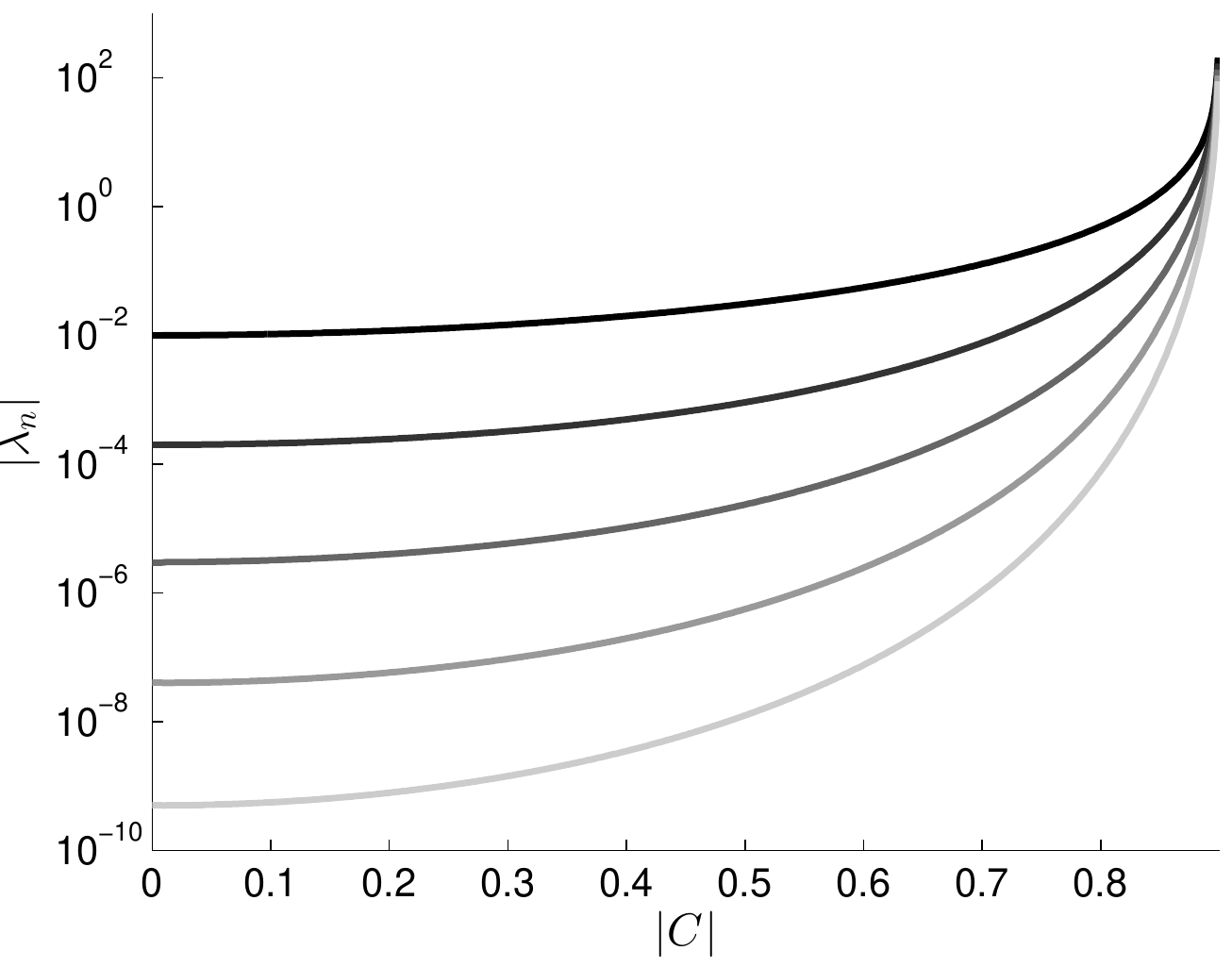}} 
\hfill
\subfloat[\label{fig:fig5b}]{\includegraphics[width=.49\textwidth]{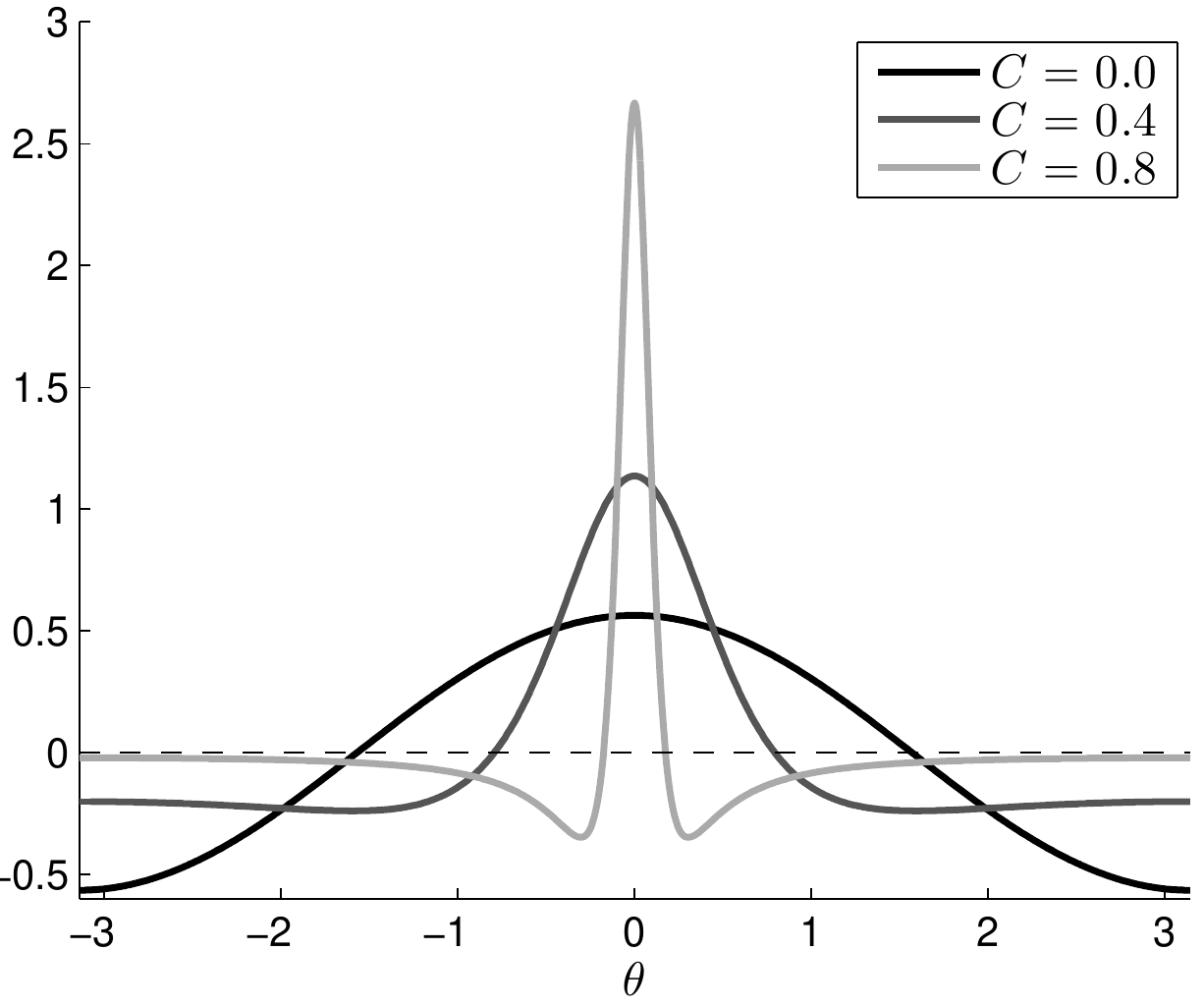}} 
\caption{\textbf{(a):} 10 largest eigenvalues $\lambda_n$ (each with multiplicity 2) of $\Lambda(\gamma_{C,R})-\Lambda(1)$ with fixed $R = 0.1$ and $0\leq \abs{C}< 1-R$. \textbf{(b):} Eigenfunction $f(\theta)$ (normalized in $\norm{\cdot}$) corresponding to the largest eigenvalue of $\Lambda(\gamma_{C,R})-\Lambda(1)$ for fixed $R = 0.1$ and various values of $C$.}
\label{fig:fig5}
\end{figure}

Figure \ref{fig:fig4} shows that for large inclusions with $r$ close to $1$ the lower bound of Theorem \ref{dnbounds} comes reasonably close, while for small inclusions with $r$ close to $0$ the upper bound is quite tight for $\rho<0.3$ (meaning inclusions close to the centre). It appears that as $r\to 0$ the distinguishability approaches a fixed curve (the curves for $r=0.1$ and $r=0.01$ are indistinguishable in the figure, and even $r=0.5$ is quite close), lying in the middle of the lower and upper bounds.

The depth dependence of EIT is further exemplified in Figure \ref{fig:fig5a} where the eigenvalues of $\Lambda(\gamma_{C,R})-\Lambda(1)$ are shown for a fixed radius $R = 0.1$ as increasing functions of the centre $|C|$. Furthermore, the eigenfunction for the largest eigenvalue is shown in Figure \ref{fig:fig5b}, and how it changes from a cosine to a very localized function as the inclusion is moved closer to the boundary. The eigenfunctions corresponding to the largest eigenvalues are the best choice of (orthonormal) boundary conditions in practice, as they maximize the distinguishability. Therefore reconstruction is expected to be more noise robust when using the eigenfunctions in the measurements. So from the behaviour in Figure \ref{fig:fig5b} it is not surprising that it is possible to numerically obtain very reasonable local reconstructions in the case of partial data (where only part of the boundary is accessible), close to the measured boundary \cite{Garde_2015,Knudsen_2015}.

\section{Conclusions} \label{sec:conc}

We have characterized the Dirichlet-to-Neumann map for ball inclusions in the unit disk (and for the Neumann-to-Dirichlet map, cf. Appendix \ref{sec:appB}), and have shown explicit lower and upper bounds on how much the distance of the inclusions to the boundary affects the operator norms. The bounds show a distinct depth dependence that can be utilized in numerical reconstruction, for instance by spatially varying regularization.

It is not known if the bounds are optimal, however through several examples it is demonstrated that the bounds accurately predict the change in distinguishability. To verify the bounds and test their tightness numerically, exact matrix representations of the boundary operators are derived, where the matrix elements are given explicitly without the need for numerical integration or solution of PDEs.

The analysis was restricted to the 2D case, though it is natural to consider if the same bounds hold for the 3D unit ball. However, in higher dimensions $d\geq 3$ the harmonic morphisms only include orthogonal transformations and translation, while M\"obius transformations generally preserve the $d$-Laplacian \cite{Manfredi_1994}. For this reason there is not a straightforward extension to 3D. 

\appendix

\section{A monotonicity property of the DN map} \label{sec:appA}

The results in this appendix are given for completeness due to a lack of proper reference.

For the Neumann-to-Dirichlet map a similar monotonicity relation as below is well-known and is used in reconstruction algorithms \cite{GardeStaboulis_2016,Harrach13,Harrach15}, where the right hand-side inequality is "flipped". In both cases of DN and ND maps the proof boils down to an application of a generalized Dirichlet principle.
\begin{lemma} \label{prop:monodn}
	Let $\gamma_1,\gamma_2\in L^\infty_+(\Omega)$ be real-valued, then
	\begin{equation*}
		\gamma_1 \leq \gamma_2 \text{ a.e.\ in } \Omega \quad\text{ implies }\quad  \inner{\Lambda(\gamma_1)f,f} \leq \inner{\Lambda(\gamma_2)f,f}, \enskip \forall f\in H^{1/2}(\partial\Omega).
	\end{equation*}
\end{lemma}
	\begin{proof}
		From the weak form of the continuum model then for any $\gamma\in L^\infty_+(\Omega)$ we have
		\begin{equation*}
			\inner{\Lambda(\gamma)f,h} = \int_{\Omega} \gamma \nabla u\cdot \overline{\nabla v}\,dx, \enskip \forall v\in H^1(\partial \Omega), v|_{\partial\Omega} = h,
		\end{equation*}
		in particular
		\begin{equation}
			\int_{\Omega} \gamma \nabla u\cdot \overline{\nabla v}\,dx = 0,\enskip \forall v\in H_0^1(\partial\Omega). \label{eq:weakh01}
		\end{equation}
		So for $v\in H_0^1(\partial\Omega)$ then \eqref{eq:weakh01} implies
		\begin{align*}
			\int_{\Omega} \gamma \abs{\nabla (u+v)}^2\,dx &= \int_{\Omega} \gamma \left( \abs{\nabla u}^2 + \abs{\nabla v}^2 + \nabla u\cdot \overline{\nabla v} + \nabla v\cdot \overline{\nabla u} \right)\, dx \\
			&= \int_{\Omega} \gamma \left( \abs{\nabla u}^2 + \abs{\nabla v}^2\right)\, dx,
		\end{align*}
		or rather
		\begin{equation}
			\inner{\Lambda(\gamma)f,f} = \int_{\Omega}\gamma \abs{\nabla u}^2\, dx =  \inf \left\{\int_{\Omega} \gamma \abs{\nabla w}^2\,dx : w\in H^1(\Omega), w|_{\partial\Omega} = f \right\}. \label{eq:lammin}
		\end{equation}
		So for any boundary potential $f\in H^{1/2}(\partial\Omega)$, and with $u_1$ being the solution to \eqref{eq:condeq} for $\gamma_1$ and $u_2$ the solution for $\gamma_2$. Then using $\gamma_1\leq \gamma_2$ in $\Omega$, and the minimizing property \eqref{eq:lammin}
		\begin{equation*}
			\inner{\Lambda(\gamma_1)f,f} = \int_{\Omega} \gamma_1\abs{\nabla u_1}^2\,dx \leq \int_{\Omega} \gamma_1\abs{\nabla u_2}^2\,dx \leq \int_{\Omega} \gamma_2\abs{\nabla u_2}^2\,dx = \inner{\Lambda(\gamma_2)f,f}.
		\end{equation*}	
	\end{proof}
	
	This leads to the very intuitive conclusion that larger inclusions gives larger distinguishability.
	
	\begin{corollary} \label{coro:mono}
		Let $A>-1$ and $D_1 \subseteq D_2 \subset \tilde{\Omega}$, where $\tilde{\Omega}\subsetneq\Omega$ such that $\dist(\tilde{\Omega},\partial\Omega)>0$, then
		\begin{equation*}
			\norm{\Lambda(1+A\chi_{D_1})-\Lambda(1)} \leq \norm{\Lambda(1+A\chi_{D_2})-\Lambda(1)}.
		\end{equation*}
	\end{corollary}
	\begin{proof}
		The case $A\equiv 0$ is trivial. Let $A > 0$ then by Lemma \ref{prop:monodn}
		\begin{align*}
			0 &= \inner{(\Lambda(1)-\Lambda(1))f,f} \\
			&\leq \inner{(\Lambda(1+A\chi_{D_1})-\Lambda(1))f,f} \\
			&\leq \inner{(\Lambda(1+A\chi_{D_2})-\Lambda(1))f,f}, \enskip\forall f\in H^{1/2}(\partial\Omega),
		\end{align*}
		and similarly if $-1<A<0$:
		\begin{align*}
			0 &= \inner{(\Lambda(1)-\Lambda(1))f,f} \\
			&\geq \inner{(\Lambda(1+A\chi_{D_1})-\Lambda(1))f,f} \\
			&\geq \inner{(\Lambda(1+A\chi_{D_2})-\Lambda(1))f,f}, \enskip\forall f\in H^{1/2}(\partial\Omega).
		\end{align*}
		Thus for any $A > -1$: 
		\begin{equation}
			\abs{\inner{(\Lambda(1+A\chi_{D_1})-\Lambda(1))f,f}} \leq \abs{\inner{(\Lambda(1+A\chi_{D_2})-\Lambda(1))f,f}}, \enskip \forall f\in H^{1/2}(\partial\Omega). \label{eq:monodiff}
		\end{equation}
		Then the claim follows directly from \eqref{eq:monodiff} and that $H^{1/2}(\partial\Omega)$ is dense in $L^2(\partial\Omega)$
		\begin{align*}
			\norm{\Lambda(1+A\chi_{D_1})-\Lambda(1)} &= \sup_{f\in H^{1/2}(\partial\Omega)\setminus\{0\}}\frac{\abs{\inner{(\Lambda(1+A\chi_{D_1})-\Lambda(1))f,f}}}{\norm{f}_{L^2(\partial\Omega)}^2}  \\
			&\leq \sup_{f\in H^{1/2}(\partial\Omega)\setminus\{0\}}\frac{\abs{\inner{(\Lambda(1+A\chi_{D_2})-\Lambda(1))f,f}}}{\norm{f}_{L^2(\partial\Omega)}^2}  \\
			&= \norm{\Lambda(1+A\chi_{D_2})-\Lambda(1)}.
		\end{align*}
	\end{proof}

\section{Distinguishability bounds and matrix characterizations for the Neumann-to-Dirichlet map} \label{sec:appB}

In this appendix we give extensions to the distinguishability bounds as well as matrix representations in terms of the Neumann-to-Dirichlet (ND) map.

The ND map is the operator $\mathcal{R}(\gamma): \nu\cdot \gamma\nabla u\mapsto u|_{\partial\Omega}$, where $u$ is the solution to the conductivity equation subject to a Neumann boundary condition $g\in H^{-1/2}_\diamond(\partial\Omega)$
\begin{equation}
	\nabla\cdot(\gamma\nabla u) = 0 \text{ in } \Omega,\quad \nu\cdot\gamma\nabla u = g\text{ on } \partial\Omega, \quad \int_{\partial\Omega} u\, ds = 0. \label{eq:pdeneu}
\end{equation}
The latter condition in \eqref{eq:pdeneu} is a grounding of the boundary potential, and is required to uniquely solve the PDE. Thus the ND map is an operator from $H^{-1/2}_\diamond(\partial\Omega)$ to $H_\diamond^{1/2}(\partial\Omega)$, where the $\diamond$-symbol indicates distributions/functions with zero mean on $\partial\Omega$. $\mathcal{R}(\gamma)$ is the inverse of $\Lambda(\gamma)$, if $\Lambda(\gamma)$ is restricted to $H^{1/2}_\diamond(\partial\Omega)$.

Returning to the domain $\Omega\equiv \rum{D}$ it is in this paper sufficient to consider $\mathcal{R}(\gamma): L^2_\diamond(\partial\Omega)\to L^2_\diamond(\partial\Omega)$ with
\[
	L^2_\diamond(\partial\Omega) \equiv \{f \in L^2(\partial\Omega) : \inner{f,1} = 0\},
\]
for which $\mathcal{R}(\gamma)$ is compact and self-adjoint (unlike the DN map where a difference of two DN maps is required for compactness).

From the proof of Lemma \ref{dntrans} we may expect that $\mathcal{R}(\gamma_{C,R}) = \mathcal{M}_a \mathcal{R}(\gamma_{0,r}) J_a^{1/2}\mathcal{M}_a$, however we need to be slightly more careful. First of all $J_a^{1/2}\mathcal{M}_a(L_\diamond^2(\partial\rum{D})) = L^2_\diamond(\partial\rum{D})$ which follows from Proposition \ref{prop:list} where the boundary integral is preserved and that $J_a^{1/2}\mathcal{M}_a$ is an involution. However, we only have $\mathcal{M}_a(L_\diamond^2(\partial\rum{D}))\subset L^2(\partial\rum{D})$. What we end up with is an ND operator from $L_\diamond^2(\partial\rum{D})$ to $\mathcal{M}_a(L_\diamond^2(\partial\rum{D}))$, corresponding to changing the grounding condition in \eqref{eq:pdeneu} to 
\[
	\int_{\partial\rum{D}} J_a^{1/2} u|_{\partial\rum{D}}\, ds = 0.
\]
Since the PDE and Neumann condition in \eqref{eq:pdeneu} gives uniqueness up to a scalar (which is chosen by the grounding condition), we can obtain the correct operator in $\mathcal{L}(L^2_\diamond(\partial\rum{D}))$ by
\begin{equation}
	\mathcal{R}(\gamma_{C,R}) = P\mathcal{M}_a \mathcal{R}(\gamma_{0,r})J_a^{1/2}\mathcal{M}_a, \label{rtrans}
\end{equation}
and similarly
\begin{equation*}
	\mathcal{R}(\gamma_{0,r}) = P\mathcal{M}_a \mathcal{R}(\gamma_{C,R})J_a^{1/2}\mathcal{M}_a, 
\end{equation*}
where $P\equiv \ident - L$ is the orthogonal projection of $L^2(\partial\rum{D})$ onto $L^2_\diamond(\partial\rum{D})$, with 
\[
	Lf \equiv \frac{1}{2\pi}\int_{\partial\rum{D}} f\, ds, \enskip f\in L^2(\partial\rum{D}).
\]
While the change is minor, the projection is necessary for the transformed ND map $\mathcal{R}(\gamma_{C,R})$ to have any eigenvalues.

\begin{proposition} \label{lemndeig}
	For $\gamma_{0,r} \equiv 1+A\chi_{B_{0,r}}$ with $0<r<1$ and $A>-1$, the eigenfunctions of $\mathcal{R}(\gamma_{0,r})$ are $f_n(\theta)\equiv \frac{1}{\sqrt{2\pi}}e^{in\theta},\enskip n\in\rum{Z}\setminus\{0\}$. The corresponding eigenvalues are
	\begin{equation*}
		\lambda_n = \frac{2+A(1-r^{2\abs{n}})}{2+A(1+r^{2\abs{n}})}\frac{1}{\abs{n}},\enskip n\neq 0.
	\end{equation*}
	The eigenvalues for the difference operator $\mathcal{R}(\gamma_{0,r})-\mathcal{R}(1)$ are
	\begin{equation}
		\lambda_n = \frac{-2Ar^{2\abs{n}}}{2+A(1+r^{2\abs{n}})}\cdot \frac{1}{\abs{n}}, \enskip n\neq 0. \label{ndlam}
	\end{equation}
	With the numbering given in \eqref{ndlam}, then $\abs{\lambda_n}$ decays monotonically with increasing $\abs{n}$.
\end{proposition}
\begin{proof}
	The eigenvalues can be derived from Proposition \ref{prop:eigval}. Now define
	\begin{equation*}
		f(x) = \frac{-2Ar^{2x}}{2+A(1+r^{2x})}\cdot\frac{1}{x}, \enskip x>0.
	\end{equation*}
	It follows immediately that 
	\begin{equation*}
		f'(x) = \frac{-2Ar^{2x}(2\log(r)x(A+2)-(A+2+Ar^{2x}))}{(A+2+Ar^{2x})^2x^2}, \enskip x>0. 
	\end{equation*}
	Since $0<r<1$ and $A>-1$ then $\log(r)<0$, $A+2>0$ and $A+2+Ar^{2x}>0$. In the case $-1<A<0$ we have $f'<0$ so $f$ is a decreasing function, however $f>0$. In the case $A>0$ then $f'>0$ so $f$ is increasing, but $f<0$. Collected we get that $\abs{f}$ is decreasing.
\end{proof}

While Proposition \ref{lemndeig} seems obvious, the corresponding case for the DN-maps does not hold for all $A$ and $r$, i.e.\ the eigenvalues for the DN-map difference does not decay monotonically with the usual numbering of the eigenvalues from the trigonometric basis.

Similar to Section \ref{sec:bnds} let $f_n(\theta) \equiv \frac{1}{\sqrt{2\pi}}e^{in\theta}$. Defining $\psi_n \equiv J_a^{1/2}\mta f_n$ makes $\{\psi_n\}_{n\in\rum{Z}\setminus\{0\}}$ an orthonormal basis for $L^2_\diamond(\partial\rum{D})$ with respect to the $\inner{\cdot,\cdot}_{-1/2}$-inner product defined in \eqref{eq:mhalfnorm}. 
\begin{theorem} \label{ndmatthm}
	Let either $H(\gamma) \equiv \mathcal{R}(\gamma)$ or $H(\gamma) \equiv \mathcal{R}(\gamma) - \mathcal{R}(1)$. Let $\hat{\lambda}_n$ be the $n$'th eigenvalue of $H(\gamma_{0,r})$ (cf. Proposition \ref{lemndeig}), and denote by $h_n$ the $n$'th Fourier coefficient of $J_a^{1/2}$ given by
	\begin{equation}
		h_n = \begin{cases} 1 & n = 0,\\ \overline{a}^{\abs{n}} & n>0,\\ a^{\abs{n}} & n<0. \end{cases} \label{Jhalffourier}
	\end{equation}
	Define the orthonormal basis $\{\psi_n\}_{n\in\rum{Z}\setminus\{0\}}$ by
	\[
		\psi_n \equiv J_a^{1/2}\mathcal{M}_a f_n, \quad f_n(\theta) \equiv \frac{1}{\sqrt{2\pi}}e^{in\theta},\enskip n\in\rum{Z}\setminus\{0\}.
	\]
	Then $H(\gamma_{C,R})$ is represented in this basis via the following matrix:
	\begin{equation}
		\mathcal{A}_{n,m} \equiv \inner{H(\gamma_{C,R})\psi_m,\psi_n}_{-1/2} = \hat{\lambda}_m(h_{n-m}-\overline{h_m}h_n),\enskip m,n\neq 0. \label{ramat}
	\end{equation}
\end{theorem}
\begin{proof}
	First the Fourier series of $J_a^{1/2}$ will be determined. Consider the case $\sang = 0$:
	\begin{equation*}
		J_\rho^{1/2}|_{e^{i\theta}} = \frac{1-\rho^2}{\abs{\rho e^{i\theta}-1}^2} = 1+\frac{\rho}{e^{-i\theta}-\rho} + \frac{\rho}{e^{i\theta}-\rho} = 1 + \sum_{n=1}^\infty\rho^ne^{in\theta}+\sum_{n=1}^\infty \rho^ne^{-in\theta},
	\end{equation*}
	where the series comes from geometric series of $\rho e^{i\theta}$ and $\rho e^{-i\theta}$, which converge as $0\leq \rho < 1$. Now $\sang\neq 0$ corresponds to a translation by $\sang$ in the $\theta$-variable:
	\begin{equation*}
		J_a^{1/2}|_{e^{i\theta}} = 1 + \sum_{n=1}^\infty\rho^ne^{in(\theta-\sang)}+\sum_{n=1}^\infty \rho^ne^{-in(\theta-\sang)} = 1 + \sum_{n=1}^\infty \overline{a}^n e^{in\theta} + \sum_{n=1}^\infty a^n e^{-in\theta},
	\end{equation*}
	which corresponds to the Fourier coefficients given in \eqref{Jhalffourier}.
	
	The adjoint of the projection operator $P$ with respect to $\inner{\cdot,\cdot}_{-1/2}$ is
	\begin{equation}
		P^* = \ident - J_a^{1/2}LJ_a^{-1/2}. \label{projadjoint}
	\end{equation}
	This follows from the calculation
	\begin{align*}
			\inner{Pf,g}_{-1/2} &= \inner{f,g}_{-1/2}-\frac{1}{2\pi}\int_{\partial \rum{D}} f\,ds\int_{\partial \rum{D}} J_a^{-1/2}\overline{g}\, ds \\
			&= \inner{f,g}_{-1/2}-\inner{f,LJ_a^{-1/2}g} \\
			&= \inner{f,(\ident-J_a^{1/2}LJ_a^{-1/2})g}_{-1/2}.
	\end{align*}

	Let $m\neq 0$, then by \eqref{rtrans} the terms of \eqref{ramat} can be expanded. Using $P^*$ from \eqref{projadjoint} and the properties in Proposition \ref{prop:list} gives
	\begin{align*}
		\mathcal{A}_{n,m} &=\inner{H(\gamma_{C,R})\psi_m,\psi_n}_{-1/2} \\
		&= \inner{P\mta H(\gamma_{0,r}) J_a^{1/2}\mta J_a^{1/2}\mta f_m,J_a^{1/2}\mta f_n}_{-1/2} \\
		&= \inner{\mta H(\gamma_{0,r}) f_m,J_a^{-1/2}(\ident-J_a^{1/2}LJ_a^{-1/2})J_a^{1/2}\mta f_n} \\
		&= \hat{\lambda}_m\inner{\mta f_m,\mta f_n} - \hat{\lambda}_m\inner{\mta f_m,L\mta f_n},
	\end{align*}
	where in the last equality it was used that $f_m$ is an eigenfunction of $H(\gamma_{0,r})$. Note that $\mta L f = L f$ as it is constant, and 
	\begin{equation*}
		L\mta f = \tfrac{1}{2\pi}\inner{\mta f,1} = \tfrac{1}{2\pi}\inner{J_a^{1/2}f,1} = LJ_a^{1/2} f.
	\end{equation*}
	Thus for $h_n = \frac{1}{\sqrt{2\pi}}\inner{J_a^{1/2},f_n} = \frac{1}{2\pi} \int_0^{2\pi}J_a^{1/2}e^{-in\theta}\,d\theta$ being the $n$'th Fourier coefficient of $J_a^{1/2}$, then
	\begin{align*}
		\mathcal{A}_{n,m} &= \frac{\hat{\lambda}_m}{\sqrt{2\pi}}\inner{J_a^{1/2} ,f_{n-m}} - \hat{\lambda}_m\inner{J_a^{1/2}f_m,LJ_a^{1/2} f_n} \\
		&= \hat{\lambda}_m h_{n-m} - \hat{\lambda}_m\inner{J_a^{1/2}f_m,1}\overline{LJ_a^{1/2}f_n} \\
		&= \hat{\lambda}_m h_{n-m} - \hat{\lambda}_m\overline{\inner{J_a^{1/2},f_m}}\frac{1}{2\pi}\inner{J_a^{1/2},f_n} \\
		&= \hat{\lambda}_m(h_{n-m}-\overline{h_m}h_n), \enskip m\neq 0. 
	\end{align*}
	Thereby concluding the proof.
\end{proof}

\begin{remark}
	The ND map can also be considered on all of $L^2(\partial\rum{D})$ by introducing the null-space $\mspan\{1\}$ such that $\mathcal{A}$ is a matrix representation of $H(\gamma_{C,R})P$ instead of $H(\gamma_{C,R})$. In that case the row $n=0$ and column $m=0$, respectively, becomes
	\begin{align*}
		\mathcal{A}_{0,m} &= \inner{H(\gamma_{C,R})P\psi_m,\psi_0}_{-1/2} = 0, \\
		\mathcal{A}_{n,0} &= \inner{H(\gamma_{C,R})P\psi_0,\psi_n}_{-1/2} = -\sum_{k\neq 0}h_k \mathcal{A}_{n,k}.
	\end{align*}
\end{remark}

Now we obtain distinguishability bounds analogous to Theorem \ref{dnbounds}.
\begin{theorem} \label{ndbounds}
	Let $\gamma$ be either $\gamma_{0,r}$ or $\gamma_{C,R}$ and denote by $\norm{\cdot}$ the operator norm on $\mathcal{L}(L^2_\diamond(\partial\rum{D}))$. From the weighted norms in \eqref{eq:halfnorm} and \eqref{eq:mhalfnorm} we have
	\begin{equation}
		\norm{\mathcal{R}(\gamma)-\mathcal{R}(1)} = \sup_{g\in L^2_\diamond(\partial\rum{D})\setminus\{0\}}\frac{\norm{(\ident-LJ_a^{1/2})(\mathcal{R}(\mta\gamma)-\mathcal{R}(1))g}_{1/2}}{\norm{g}_{-1/2}}. \label{ndnorms}
	\end{equation}
	Furthermore the following bounds hold: 
	\begin{equation}
		\frac{1-\rho}{1+\rho}\norm{\mathcal{R}(\gamma_{C,R}) - \mathcal{R}(1)} \leq \norm{\mathcal{R}(\gamma_{0,r}) - \mathcal{R}(1)} \leq \frac{\sqrt{1+\rho^2}}{1-\rho^2}\norm{\mathcal{R}(\gamma_{C,R}) - \mathcal{R}(1)}. \label{eq:ndbnds}
	\end{equation}
\end{theorem}
\begin{proof}
By \eqref{rtrans}
\begin{align*}
	\norm{\mathcal{R}(\gamma)-\mathcal{R}(1)}^2 &= \sup_{g\in L^2_\diamond(\partial\rum{D})\setminus\{0\}} \frac{\norm{P\mta (\mathcal{R}(\mathcal{M}_a\gamma)-\mathcal{R}(1))J_a^{1/2}\mta g}^2}{\norm{g}^2} \\
	&= \sup_{g\in L^2_\diamond(\partial\rum{D})\setminus\{0\}}\frac{\int_{\partial\rum{D}} \abs{P\mta (\mathcal{R}(\mathcal{M}_a\gamma)-\mathcal{R}(1))J_a^{1/2}\mta g}^2\,ds}{\int_{\partial\rum{D}}J_a^{1/2}\abs{\mta g}^2\,ds}.
\end{align*}
Utilizing that $J_a^{1/2}\mta(L^2_\diamond(\partial\rum{D})) = L^2_\diamond(\partial\rum{D})$, we can substitute $J_a^{1/2}\mta g$ with $g$, and afterwards use that $P\mta = \mta-LJ_a^{1/2}$
\begin{align*}
	\norm{\mathcal{R}(\gamma)-\mathcal{R}(1)}^2 &= \sup_{g\in L^2_\diamond(\partial\rum{D})\setminus\{0\}}\frac{\int_{\partial\rum{D}} \abs{P\mta(\mathcal{R}(\mathcal{M}_a\gamma)-\mathcal{R}(1))g}^2\,ds}{\int_{\partial\rum{D}}J_a^{-1/2}\abs{g}^2\,ds} \\
	&= \sup_{g\in L^2_\diamond(\partial\rum{D})\setminus\{0\}}\frac{\int_{\partial\rum{D}} \abs{(\mta-LJ_a^{1/2})(\mathcal{R}(\mathcal{M}_a\gamma)-\mathcal{R}(1))g}^2\,ds}{\int_{\partial\rum{D}}J_a^{-1/2}\abs{g}^2\,ds}.
\end{align*} 
Applying the change of variables $\mta$ and $\mta L = L$ yields the expression in \eqref{ndnorms}
\begin{align*}
	\norm{\mathcal{R}(\gamma)-\mathcal{R}(1)}^2 &= \sup_{g\in L^2_\diamond(\partial\rum{D})\setminus\{0\}}\frac{\int_{\partial\rum{D}} J_a^{1/2}\abs{(\ident-LJ_a^{1/2})(\mathcal{R}(\mathcal{M}_a\gamma)-\mathcal{R}(1))g}^2\,ds}{\int_{\partial\rum{D}}J_a^{-1/2}\abs{g}^2\,ds} \\
		&= \sup_{g\in L^2_\diamond(\partial\rum{D})\setminus\{0\}}\frac{\norm{(\ident-LJ_a^{1/2})(\mathcal{R}(\mathcal{M}_a\gamma)-\mathcal{R}(1))g}_{1/2}^2}{\norm{g}_{-1/2}^2}.
\end{align*}
Now let $\hat{g}_1 \equiv e^{i\theta}$ which by Proposition \ref{lemndeig} is the eigenfunction corresponding to the largest eigenvalue $\hat{\lambda}_1$ for $\mathcal{R}(\gamma_{0,r})-\mathcal{R}(1)$. Let $\lambda_1$ be the largest eigenvalue for $\mathcal{R}(\gamma_{C,R})-\mathcal{R}(1)$, then \eqref{ndnorms} implies
\begin{align}
	\abs{\lambda_1}^2 &= \norm{\mathcal{R}(\gamma_{C,R})-\mathcal{R}(1)}^2 \label{ndbndmideq}\\
	&= \sup_{g\in L^2_\diamond(\partial\rum{D})\setminus\{0\}}\frac{\int_{\partial\rum{D}} J_a^{1/2}\abs{(\ident-LJ_a^{1/2})(\mathcal{R}(\gamma_{0,r})-\mathcal{R}(1))g}^2\,ds}{\int_{\partial\rum{D}}J_a^{-1/2}\abs{g}^2\,ds} \notag \\
	& \geq \abs{\hat{\lambda}_1}^2\frac{\int_{\partial\rum{D}} J_a^{1/2}\abs{(\ident-LJ_a^{1/2})\hat{g}_1}^2\,ds}{\int_{\partial\rum{D}}J_a^{-1/2}\,ds} \notag\\
	&= \frac{\abs{\hat{\lambda}_1}^2}{2\pi}\frac{1-\rho^2}{1+\rho^2}\int_{\partial\rum{D}} J_a^{1/2}\abs{(\ident-LJ_a^{1/2})\hat{g}_1}^2\,ds, \label{eq:ndbnd1}
\end{align}
where the integral of $J_a^{-1/2}$ was calculated in \eqref{Jmhalfint}. Expanding $J_a^{1/2}$ in its Fourier series from \eqref{Jhalffourier} gives $J_a^{1/2}\hat{g}_1 = \sum_{k\in\rum{Z}} h_k e^{i(k+1)\theta}$, thus
\begin{equation}
	LJ_a^{1/2}\hat{g}_1 = \frac{1}{2\pi}\int_{\partial\rum{D}} J_a^{1/2}\hat{g}_1\,ds = \frac{1}{2\pi}\sum_{k\in\rum{Z}}h_k\int_{0}^{2\pi} e^{i(k+1)\theta}\,d\theta = h_{-1} = a. \label{eq:ndbnd2}
\end{equation}	
By inserting \eqref{eq:ndbnd2} into \eqref{eq:ndbnd1}, again applying the Fourier series of $J_a^{1/2}$ from \eqref{Jhalffourier} and that $\int_{\partial\rum{D}} J_a^{1/2}\,ds = 2\pi$ gives the upper bound
\begin{align*}
	\abs{\lambda_1}^2 &\geq \frac{\abs{\hat{\lambda}_1}^2}{2\pi}\frac{1-\rho^2}{1+\rho^2}\int_{0}^{2\pi} J_a^{1/2}\abs{e^{i\theta}-a}^2\,d\theta \\
	&= \frac{\abs{\hat{\lambda}_1}^2}{2\pi}\frac{1-\rho^2}{1+\rho^2}\int_{0}^{2\pi} J_a^{1/2}(1+\rho^2-\overline{a}e^{i\theta}-ae^{-i\theta})\,d\theta \\
	&=  \frac{\abs{\hat{\lambda}_1}^2}{2\pi}\frac{1-\rho^2}{1+\rho^2} \left[2\pi(1+\rho^2)-\int_{0}^{2\pi}\left(\overline{a}\sum_{k\in\rum{Z}}h_ke^{i(k+1)\theta} + a\sum_{k\in\rum{Z}}h_ke^{i(k-1)\theta}\right)\,d\theta\right] \\
	&= \frac{\abs{\hat{\lambda}_1}^2}{2\pi}\frac{1-\rho^2}{1+\rho^2}\left[2\pi(1+\rho^2)-2\pi(\overline{a}h_{-1}+ah_1)\right] \\
	&= \abs{\hat{\lambda}_1}^2\frac{(1-\rho^2)^2}{1+\rho^2}.
\end{align*}
Thus
\begin{equation*}
	\norm{\mathcal{R}(\gamma_{0,r})-\mathcal{R}(1)} = \abs{\hat{\lambda}_1} \leq \frac{\sqrt{1+\rho^2}}{1-\rho^2}\abs{\lambda_1} = \frac{\sqrt{1+\rho^2}}{1-\rho^2}\norm{\mathcal{R}(\gamma_{C,R})-\mathcal{R}(1)}.
\end{equation*}
Now consider the opposite case for \eqref{ndbndmideq}, and let $g_1$ be a normalized (in $\norm{\cdot}_{L^2(\partial\rum{D})}$) eigenfunction corresponding to the largest eigenvalue $\lambda_1$ of $\mathcal{R}(\gamma_{C,R})-\mathcal{R}(1)$. Using the bounds \eqref{eq:normeqv}
\begin{align*}
\abs{\hat{\lambda}_1}^2 &= \norm{\mathcal{R}(\gamma_{0,r})-\mathcal{R}(1)}^2 \\
&= \sup_{g\in L^2_\diamond(\partial\rum{D})\setminus\{0\}}\frac{\norm{ (\ident-LJ_a^{1/2})(\mathcal{R}(\gamma_{C,R})-\mathcal{R}(1))g}^2_{1/2}}{\norm{g}^2_{-1/2}}  \\
&\geq \abs{\lambda_1}^2\frac{\norm{(\ident-LJ_a^{1/2})g_1}_{1/2}^2}{\norm{g_1}_{-1/2}^2}\\
&\geq \abs{\lambda_1}^2\left(\tfrac{1-\rho}{1+\rho}\right)^2\norm{(\ident-LJ_a^{1/2})g_1}^2.
\end{align*}
Now utilizing that $g_1\in L_\diamond^2(\partial\rum{D})$, so as $LJ_a^{1/2}g_1$ is constant then $\inner{LJ_a^{1/2}g_1,g_1} = 0$:
\begin{align}
	\abs{\hat{\lambda}_1}^2 &\geq \abs{\lambda_1}^2\left(\tfrac{1-\rho}{1+\rho}\right)^2(\norm{g_1}^2 +\norm{LJ_a^{1/2}g_1}^2) \notag\\
	&= \abs{\lambda_1}^2\left(\tfrac{1-\rho}{1+\rho}\right)^2(1 +2\pi\abs{LJ_a^{1/2}g_1}^2) \label{betterlowbnd} \\
	&\geq \abs{\lambda_1}^2\left(\tfrac{1-\rho}{1+\rho}\right)^2, \notag
\end{align} 
which gives the lower bound in \eqref{eq:ndbnds}.
\end{proof}

Numerically it can be verified (cf. Figure \ref{fig:fig6a}) that 
\begin{equation*}
	\norm{\mathcal{R}(\gamma_{C,R})-\mathcal{R}(1)} \leq \norm{\mathcal{R}(\gamma_{0,r})-\mathcal{R}(1)},
\end{equation*}
which is a stronger bound than in Theorem~\ref{ndbounds}. However, in the proof even the bound \eqref{betterlowbnd} which depends on $g_1$ does not give $\norm{\mathcal{R}(\gamma_{C,R})-\mathcal{R}(1)} \leq \norm{\mathcal{R}(\gamma_{0,r})-\mathcal{R}(1)}$ in general. 

\begin{remark}
	It is possible to remove the projection operator $P$ in Theorem \ref{ndbounds}, which led to its lengthy proof, by writing the norm as
	\begin{equation*}
		\norm{\mathcal{R}(\gamma)-\mathcal{R}(1)} = \sup_{g\in L^2_\diamond(\partial\rum{D})\setminus\{0\}}\frac{\abs{\inner{P\mathcal{M}_a(\mathcal{R}(\mta\gamma)-\mathcal{R}(1))J_a^{1/2}\mta g,g}}}{\norm{g}^2},
	\end{equation*}
	and abusing that $P$ is self-adjoint in the usual $L^2(\partial\rum{D})$-inner product (as it is an orthogonal projection). The proof would give the same lower bound, however it leads to the worse upper bound with the term $(1+\rho^2)/(1-\rho^2)$ instead of $\sqrt{1+\rho^2}/(1-\rho^2)$.
\end{remark}

\begin{figure}[htb]
\centering
\subfloat[\label{fig:fig6a}]{\includegraphics[width=.49\textwidth]{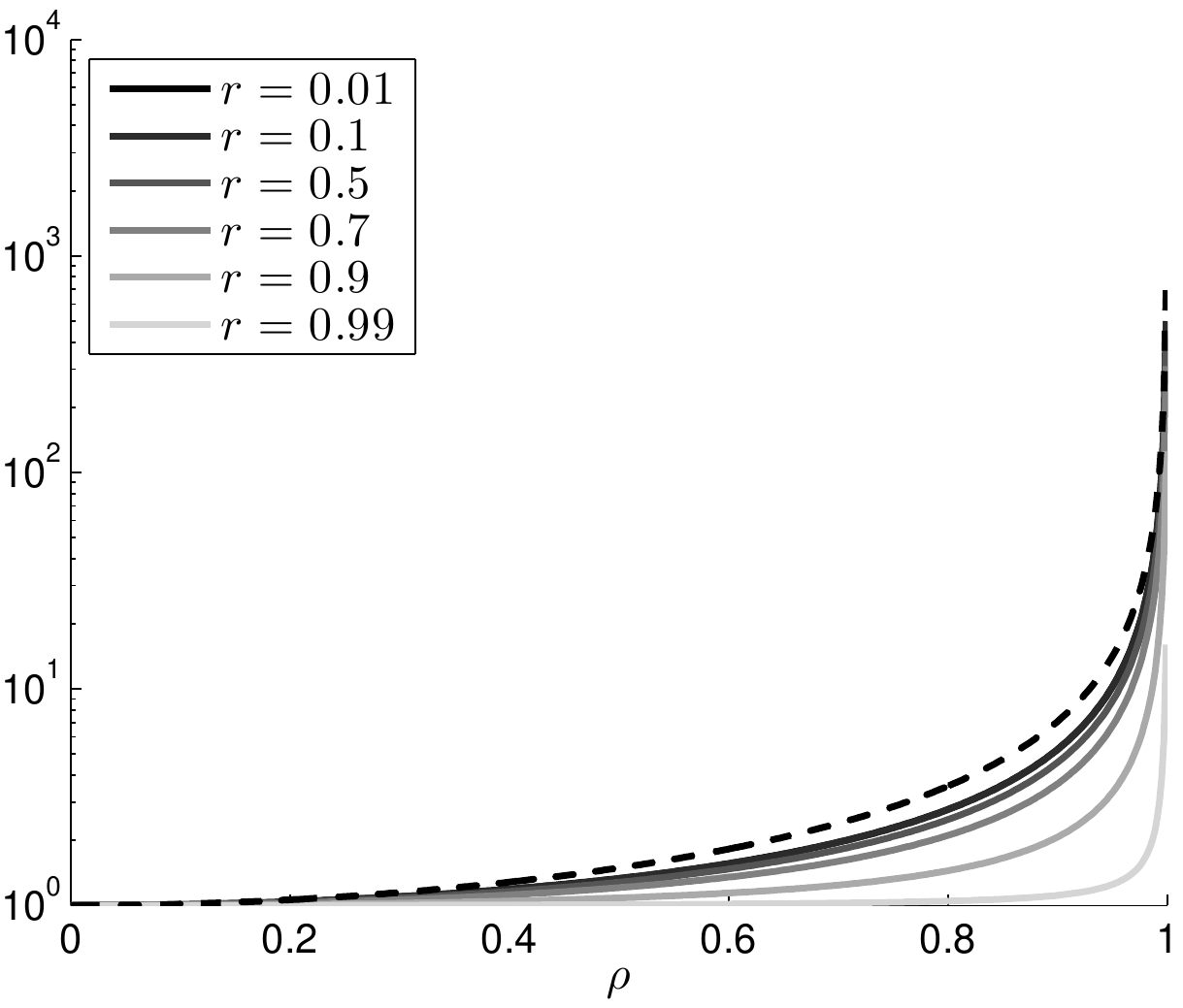}}
\hfill
\subfloat[\label{fig:fig6b}]{\includegraphics[width=.49\textwidth]{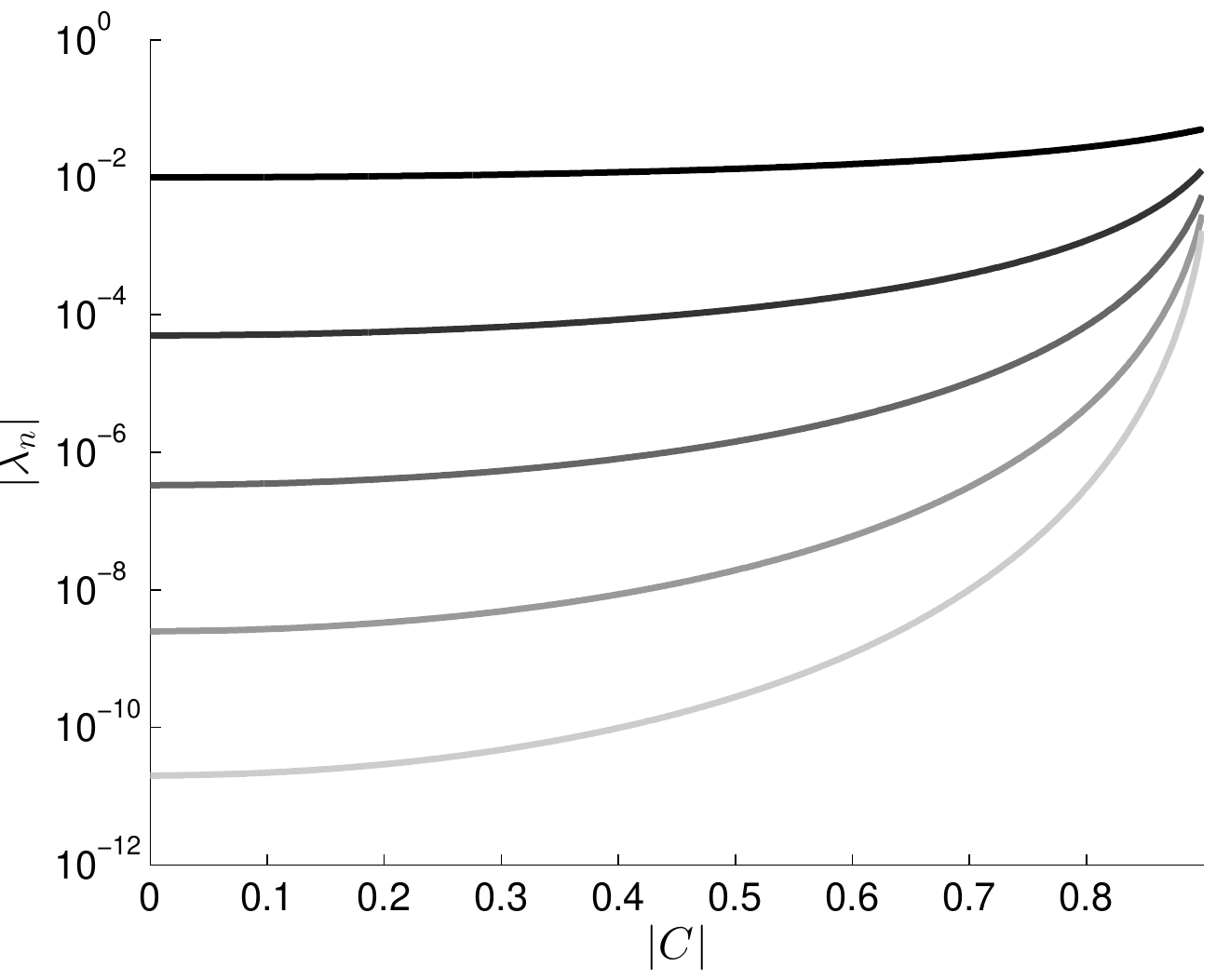}} 
\caption{\textbf{(a):} Ratio $\norm{\mathcal{R}(\gamma_{0,r})-\mathcal{R}(1)}/\norm{\mathcal{R}(\gamma_{C,R})-\mathcal{R}(1)}$ for $\abs{a} = \rho\in[0,1)$ where $R$ and $C$ are determined from $r$ and $\rho$ by Proposition \ref{prop:mobdisk}, along with the upper bound (dashed line) from Theorem \ref{ndbounds}. \textbf{(b):} 10 largest eigenvalues $\lambda_n$ (each with multiplicity 2) of $\mathcal{R}(\gamma_{C,R})-\mathcal{R}(1)$ with fixed $R = 0.1$ and $0\leq \abs{C}< 1-R$.}
\label{fig:fig6}
\end{figure}

Figure \ref{fig:fig6a} shows that the upper bound in Theorem \ref{ndbounds} is very reasonable for small inclusions with $r$ close to $0$. Furthermore, it shows (for the chosen examples) that the distinguishability is decreasing as $\rho$ is increased, meaning $\norm{\mathcal{R}(\gamma_{C,R})-\mathcal{R}(1)} \leq \norm{\mathcal{R}(\gamma_{0,r})-\mathcal{R}(1)}$. This is different from what was observed for the DN map in Figure~\ref{fig:fig4}, however it is worth noting that the radius $R$ is decreasing with $\rho$, and in Figure~\ref{fig:fig6b} where the radius is kept fixed, the distinguishability is increasing. Thus, for the ND map the distinguishability is increasing at a slower rate as the distance to the boundary is reduced (compared to the DN map), and is not able to overcome the change in radius from $r$ to $R$. It is therefore worth noting that reconstruction based on ND- and DN-maps are fundamentally different in terms of depth dependence. 

\bibliographystyle{siam}
\bibliography{minbib}

\end{document}